 \documentclass[12pt]{amsart}

\usepackage{amssymb, amsmath, amsfonts}
\usepackage[raiselinks=false,colorlinks=true,citecolor=blue,urlcolor=blue,
linkcolor=blue,bookmarksopen=true,pdftex]{hyperref}
 
 \usepackage{enumerate}
\usepackage[mathscr]{eucal}
\newtheorem{theorem}{Theorem}[section]
\newtheorem{proposition}[theorem]{Proposition}
\newtheorem{lemma}[theorem]{Lemma}
\newtheorem{corollary}[theorem]{Corollary}
\newtheorem{remark}[theorem]{Remark}
\newtheorem{definition}[theorem]{Definition}
\newtheorem{example}[theorem]{Example}

\renewcommand{\hom}{\textrm{Hom}}
\newcommand{\wt}{\widetilde}

\renewcommand{\span}{\textrm{span}}

\newcommand{\flag}{\textrm{Flag}}
 
\newcommand{\fix}{\textrm{Fix}}
\begin{document}
\baselineskip=15.5pt
\title[Generalized Dold manifolds]{On generalized Dold manifolds} 
\author[A. Nath]{Avijit Nath} 
\author[P. Sankaran]{Parameswaran Sankaran}
\address{Institute of Mathematical Sciences, (HBNI), 
CIT Campus, Taramani, Chennai 600113}

\email{avijitnath@imsc.res.in}
\email{sankaran@imsc.res.in}
\subjclass[2010]{57R25, 57R20.}
\keywords{Dold manifolds, flag manifolds, Stiefel-Whitney classes, stable parallelizability, cobordism}
\thispagestyle{empty}
\date{}
\dedicatory{Dedicated to Professor D. S. Nagaraj on the occasion of his sixtieth birthday}
\thanks{Both authors were partially supported by a XII Plan Project, Department of Atomic Energy, Government of India.}
\begin{abstract}
Let $X$ be a smooth manifold with a (smooth) involution $\sigma:X\to X$ such that $\fix(\sigma)\ne \emptyset$.
We call the space $P(m,X):=\mathbb{S}^m\times X/\!\sim$ where $(v,x)\sim (-v,\sigma(x))$ a generalized 
Dold manifold.  When $X$ is an almost complex manifold and the differential $T\sigma: TX\to TX$ is 
conjugate complex linear on each fibre, we obtain a formula for the Stiefel-Whitney polynomial 
of $P(m,X)$ when $H^1(X;\mathbb{Z}_2)=0$.  We obtain results on stable parallelizability of $P(m,X)$ 
and a very general criterion for the (non) vanishing of the unoriented cobordism class  $[P(m,X)]$ 
in terms of the corresponding properties for $X$.   These results are applied to the case when $X$ 
is a complex flag manifold.

\end{abstract}

\maketitle

\section{Introduction} \label{intro}

Let $P(m,n)$ denote the space obtained as the quotient by the cyclic group $\mathbb{Z}_2$-action on the product 
$\mathbb{S}^m\times \mathbb{C}P^n$ generated by the involution $(u,L)\mapsto (-u,\bar L), u\in \mathbb{S}^m, L\in \mathbb{C}P^n$ 
where $\bar{L}$ denotes the complex conjugation.  
The spaces $P(m,n)$, which seem to have first appeared in the work of Wu, 
are called Dold manifolds, after it was shown by 
Dold \cite{dold} that, for suitable values of $m,n$, the cobordism classes 
of $P(m,n)$ serve as generators in odd degrees for the unoriented 
cobordism algebra $\mathfrak{N}$.  Dold manifolds have been extensively studied and have received renewed attention in recent years; see \cite{korbas}, \cite{novotny} and also \cite{nt}, \cite{thakur}, and \cite{ct}.

The construction of Dold manifolds suggests, among others, the following generalization.
Consider an involution on a Hausdorff topological space  $\sigma:X\to X$ with 
non-empty fixed point set and consider the space $P(m,X,\sigma)$ obtained as the quotient of $\mathbb{S}^m\times X$ by the action of $\mathbb{Z}_2$ defined by the fixed point free involution $(v,x)\mapsto (-v,\sigma(x))$.   We obtain a locally 
trivial fibre bundle with projection $\pi: P(m,X,\sigma)\to \mathbb{R}P^m$ and fibre space $X$.  If $x_0$ 
is a fixed point of $\sigma$, then the bundle 
admits a cross-section $s:\mathbb{R}P^m\to P(m,X,\sigma)$ defined as $s([v])=[v,x_0]$.  
If $X$ is a smooth manifold and if $\sigma$ is smooth, then the above bundle and the cross-section are smooth.  

In this paper we study certain manifold-properties of $P(m,X,\sigma)$  (or more briefly $P(m,X)$) where $X$ is a closed connected smooth manifold 
with an almost complex structure $J:TX\to TX$ and $\sigma$ is a conjugation, that is, the differential 
$T\sigma: TX\to TX$ and $J$ anti-commute: $T\sigma\circ J=-J\circ T\sigma$.   
We give a description of the tangent bundle of $P(m,X)$.  Assuming that $\fix(\sigma)\ne \emptyset$ and $H^1(X;\mathbb{Z}_2)=0$, we 
obtain a formula for the Stiefel-Whitney classes of $P(m,X)$ (Theorem \ref{swformulaforp}) and a necessary and sufficient condition 
for $P(m,X)$ to admit a spin structure (Theorem \ref{spin}).   
We also obtain results on the stable parallelizability of the $P(m,X)$  (Theorem \ref{stableparallel}) and  the vanishing of their (unoriented) cobordism class in the cobordism ring $\mathfrak{N}$ (Theorem \ref{null-nonorient}).     

Recall that a smooth manifold $M$ is said to be parallelizable (resp. stably parallelizable) if its tangent bundle 
$\tau M$ (resp. $\epsilon_\mathbb{R}\oplus \tau M$) is trivial.   

By the celebrated work of Adams \cite{adams} on the vector field problem for spheres, one knows that 
the (additive) order of the element $([\zeta]-1)\in KO(\mathbb{R}P^m)$ equals $2^{\varphi(m)}$ where 
$\zeta$ is the Hopf line bundle over $\mathbb{R}P^{m}$ and 
$\varphi(m)$ is the number of positive integers $j\le m$ such that $j\equiv 0,1,2,$ or $4\mod 8$.

The complex flag manifold $\mathbb{C}G(n_1,\ldots, n_r)$ is the homogeneous space $U(n)/(U(n_1)\times\cdots\times U(n_r))$, 
where the $n_j\ge 1$ are positive integers and $n=\sum_{1\le j\le r} n_j$.  These manifolds  are well-known to be complex 
projective varieties.   
We denote by $P(m; n_1,\ldots,n_r)$ the space $P(m,\mathbb{C}G(n_1,\ldots,n_r))$.   
The complete flag manifold $\mathbb{C}G(1,\ldots,1)$ is denoted $\flag(\mathbb{C}^n)$.   
Note that $\mathbb{C}G(n_1,n_2)$ is the complex Grassmann manifold $\mathbb{C}G_{n,n_1}$ 
of $n_1$-dimensional vector subspaces of $\mathbb{C}^{n}$.

We highlight here the results on stable parallelizability and cobordism for a restricted 
classes of generalized Dold manifolds as in these cases the results 
are nearly complete.

\begin{theorem} \label{parallel-grassmann}  Let $m\ge 1$ and $r\ge 2$. \\
(i) The manifold $P(m;n_1,\ldots, n_r)$ is stably parallelizable if and only if  $n_j=1$  for all $j$ and 
$ 2^{\varphi(m)}$ 
divides $(m+1+{n\choose 2})$.  \\ (ii) Suppose that $P:=P(m;1,\ldots,1)$ is stably parallelizable. Then it is 
parallelizable if $\rho(m+1)>\rho(m+1+n(n-1))$.  If $m$ is even, then $P$ is not parallelizable.
 \end{theorem}

The case when the flag manifold is a complex projective space corresponds to the classical Dold manifold $P(m,n-1).$ 
 In this special case the above result is  
due to J. Korba\v{s} \cite{korbas}.   See 
also \cite{ucci} in which J. Ucci characterized classical Dold manifolds which admit codimension-one 
embeddings in the Euclidean space. 


\begin{theorem} \label{null-grassmann}
Let $1\le k\le n/2$ and let $m\ge 1$.\\(i) If 
$\nu_2(k)<\nu_2(n)$, then $[P(m,\mathbb{C}G_{n,k})]=0$ in $\mathfrak{N}$. \\
(ii)  If $m\equiv 0\mod 2$ and if 
$\nu_2(k)\ge \nu_2(n)$,  then $[P(m,\mathbb{C}G_{n,k})]\ne 0$.
\end{theorem}
The above theorem leaves out the case when $m\ge 1$ is odd and $\nu_2(k)\ge \nu_2(n)$.   See Remark \ref{null-flag} 
for results on the vanishing of $[P(m;n_1,\ldots,n_r)]$.    

Our proofs make use of basic concepts in the theory of vector bundles and characteristic classes.  
We first introduce, in \S2, the notion of a $\sigma$-conjugate complex vector bundle over $X$ where 
$\sigma$ is an involution on $X$ and associate to each such complex vector bundle $\omega$ a 
real vector bundle over $\hat\omega$.  We establish a splitting principle to obtain a formula 
for the Stiefel-Whitney classes of $\hat \omega$ in terms of certain `cohomology extensions' of 
Stiefel-Whitney classes of $\omega$, assuming 
that $H^1(X;\mathbb{Z}_2)=0$.   This leads to a formula for the Stiefel-Whitney classes of $P(m,X)$ 
when $X$ is a smooth almost complex manifold and $\sigma$ is a complex conjugation. 
Proof of Theorem \ref{parallel-grassmann} uses the main result of \cite{sz}, the Bredon-Kosi\'nski's theorem \cite{bk}, and a 
certain functor $\mu^2$ introduced by Lam \cite{lam} to study immersions of flag manifolds.
Proof of Theorem \ref{null-grassmann} uses basic facts from the theory of 
Clifford algebras, a result of Conner and Floyd \cite[Theorem 30.1]{cf} concerning cobordism of manifolds admitting 
stationary point free action of elementary abelian $2$-group, and the main theorem of \cite{sankaran}.


\section{Vector bundles over $P(m,X,\sigma)$}
Let $\sigma:X\to X$ be an involution of a path connected paracompact Hausdorff topological space and let $\omega$ be a complex vector bundle 
over $X$.  Denote by $\omega^\vee$ the dual vector bundle $\hom_\mathbb{C}(\omega, \epsilon_\mathbb{C})$.  Here $\epsilon_\mathbb{F}$ denotes the  
the trivial $\mathbb{F}$-line bundle over $X$ where $\mathbb{F}=\mathbb{R}, \mathbb{C}$.  Note that, since $X$ is paracompact, $\omega$ admits a Hermitian metric and 
so $\omega^\vee$ is isomorphic to the conjugate bundle $\bar{\omega}$.  The following definition 
generalises the notion of a conjugation of an almost complex manifold in the sense of Conner and Floyd \cite[\S 24]{cf}. 

\begin{definition} \label{sigmaconjugate}
Let $\sigma:X\to X$ be an involution and let $\omega$ be a complex vector 
bundle over $X$.  A $\sigma$-{\em conjugation} on $\omega$ is 
an involutive bundle map  
$\hat{\sigma}:E(\omega)\to E(\omega)$ that 
covers $\sigma$ which is conjugate complex linear on the fibres of $\omega$.  If such a $\hat\sigma$ exists, we say that 
$(\omega,\hat\sigma)$ (or more briefly $\omega$) is a $\sigma$-{\em conjugate bundle}.  
\end{definition}

Note that if $\omega$ is a $\sigma$-conjugate bundle, then $\bar{\omega}\cong \sigma^*(\omega)$.  

\begin{example}\label{basic}
{\em 
(i)  Let $\sigma$ be any involution on $X$.
When $\omega=n\epsilon_\mathbb{C}$, the trivial complex vector bundle of rank $n$, we have $E(\omega)=X\times \mathbb{C}^n$.  The {\it standard} $\sigma$-conjugation on $\omega$ is defined as  
$\hat\sigma(x, \sum z_je_j)=(\sigma(x), \sum \bar{z}_je_j)$.   Here $\{e_j\}_{1\le j\le n}$ is the standard basis of $\mathbb{C}^n$.  Thus $(n\epsilon_\mathbb{C},\hat\sigma)$ is $\sigma$-conjugate bundle. 

(ii)  Let $X=\mathbb{C}G_{n,k}$ and let $\sigma:X\to X$ be the involution $L\mapsto \bar{L}$.   Then the standard $\sigma$-conjugation on $n\epsilon_\mathbb{C}$ defines, by restriction, a $\sigma$-conjugation 
of the canonical 
$k$-plane bundle $\gamma_{n,k}$.  Explicitly, $ v\mapsto \bar{v}, ~v\in L\in \mathbb{C}G_{n,k},$ 
is the required involutive bundle map $\hat\sigma:E(\gamma_{n,k})\to E(\gamma_{n,k})$ that covers $\sigma$.  
Similarly the orthogonal complement $\beta_{n,k}:=\gamma_{n,k}^\perp$ is also a $\sigma$-conjugate bundle.

(iii) If $X\subset \mathbb{C}P^N$ is a complex projective manifold defined over $\mathbb{R}$ and $\sigma:X\to X$ is the 
restriction of complex conjugation $[z]\mapsto [\bar{z}]$, then the tangent bundle 
$\tau X$ of $X$ is a $\sigma$-conjugate bundle.  Indeed the differential of $\sigma$, namely $T\sigma: TX\to TX$ is the required bundle map 
$\hat{\sigma}$ of $\tau X$ that covers $\sigma$.  As mentioned above, this classical case was generalized by Conner and Floyd  \cite[\S 24]{cf} to the case when $X$ is an almost complex 
manifold. 

(iv)  If $\omega, \eta$ are $\sigma$-conjugate vector bundles over $X$, then so are $\Lambda^r(\omega), 
\hom_\mathbb{C}(\omega,\eta), \omega\otimes \eta$, and $\omega\oplus \eta$.  
For example, if $\hat\sigma$ and $\tilde \sigma$ are $\sigma$-conjugations on $\omega$ and $\eta$ respectively, 
both covering $\sigma$,  then 
$\hom_\mathbb{C}(\omega, \eta)\ni f\mapsto \tilde \sigma\circ f 
\circ \hat\sigma \in \hom_\mathbb{C}(\omega,\eta)$ is verified to be a conjugate complex linear bundle involution of 
$\hom_\mathbb{C}(\omega,\eta)$
that covers $\sigma$.

(v) Any subbundle $\eta$ of a $\sigma$-conjugate complex vector bundle $\omega$ over $X$ is also $\sigma$-conjugate 
provided $\hat\sigma:E(\omega)\to E(\omega)$ satisfies $\hat\sigma(E(\eta))=E(\eta).$ }
\end{example}

\subsection{Vector bundle associated to $(\eta,\hat{\sigma})$} \label{assocbundle}
Let $\eta $ be a {\it real} vector bundle over $X$ with projection $p_\eta:E(\eta)\to X$ 
and let $\hat{\sigma}: E(\eta)\to E(\eta)$ be an involutive 
bundle isomorphism that covers $\sigma$.   We obtain a real vector bundle, denoted $\hat{\eta},$ over 
$P(m,X,\sigma)$ as follows:  $(v,e)\mapsto (-v,\hat\sigma(e))$ defines a fixed point free involution of 
$\mathbb{S}^m\times E(\eta)$ with orbit space $P(m,E(\eta),\hat\sigma)$.  The map 
$p_{\hat\eta}: P(m,E(\eta),\hat\sigma)\to P(m,X,\sigma)$ defined as 
$[v,e]\mapsto [v,p_\eta(e)]$ is the projection of the required bundle $\hat\eta$.  

This construction is 
applicable when $\eta=\rho(\omega)$, the underlying real vector bundle of a $\sigma$-conjugate 
complex vector bundle 
$(\omega,\hat \sigma)$.  
If $\beta$ is a (real) subbundle of $\eta$ such that $\hat\sigma(E(\beta))=E(\beta)$, then the restriction of 
$\hat\sigma$ to $E(\beta)$ defines a bundle $\hat \beta$ which is evidently a subbundle   
of $\hat\eta$.

We shall denote by $\xi$ the real line bundle over $P(m,X,\sigma)$, often referred to as the Hopf bundle, associated to the double 
cover $\mathbb{S}^m\times X\to P(m,X,\sigma)$.  Its total space has the description 
$\mathbb{S}^m\times X\times_{\mathbb{Z}_2} \mathbb{R}$ consisting of elements $[v,x,t]=\{(v,x,t), (-v,\sigma(x), -t)\},  v\in \mathbb{S}^m,x\in X, t\in 
\mathbb{R}$.   
Denote by $\pi:P(m,X,\sigma)\to \mathbb{R}P^m$ the map $[v,x]\mapsto [v]$.  Then $\pi$ is the projection 
of a fibre bundle with fibre $X$. The map $E(\xi)\to E(\zeta)$ defined as $[v,x,t]\mapsto [v,t]$ is a bundle map 
that covers the projection $\pi:P(m,X,\sigma)\to \mathbb{R}P^m$ and so $\xi\cong \pi^*(\zeta)$.  

 If $\sigma(x_0)=x_0\in X$, then we have a cross-section $s:\mathbb{R}P^m\to P(m,X)$ 
defined as $[v]\mapsto [v,x_0]$.  Note that $s^*(\xi)=\zeta$.  

\subsection{Dependence of $\hat\omega$ on $\hat\sigma$}\label{dependence}

{\em It should be noted that the definition of $\hat\eta$ depends not only on the real vector bundle $\eta$ but 
also on the bundle map $\hat\sigma$ that covers $\sigma$.}  For example, on the trivial line bundle 
$\epsilon_\mathbb{R}$, if $\hat\sigma(x,t)=(\sigma(x),t)$, then $\hat\epsilon_\mathbb{R}\cong \epsilon_\mathbb{R}$, 
whereas if $\hat\sigma(x,t)=(\sigma(x),-t)$, then $\hat\epsilon_\mathbb{R}$ is isomorphic to $\xi$.  

When $\omega=\tau X$ is the tangent bundle over an almost complex manifold $(X,J)$ and $\hat\sigma=T\sigma$ 
where $\sigma$ is a conjugation on $X$, (i.e., 
satisfies $J_{\sigma(x)}\circ T_x\sigma=-T_{x}\sigma\circ J_x~\forall x\in X$), the vector bundle 
$\hat\tau X$ is understood to be defined with respect to the pair $(\tau X,T\sigma)$.  

Let $k,l\ge 0$ be integers and let $n=k+l\ge 1$ and let $s_1,\ldots, s_n$ be everywhere linearly independent 
sections of the trivial bundle $n\epsilon_\mathbb{R}$.
Denote by $\varepsilon_{k,l}:X\times \mathbb{R}^n\to X\times \mathbb{R}^n$ the involutive bundle map 
$n\epsilon_{\mathbb{R}}$ covering $\sigma$ defined as 
$\varepsilon_{k,l}(x,\sum_jt_js_j(x))=(\sigma(x),-\sum_{1\le j\le k}t_js_j(x)+\sum_{k<j\le n}t_js_j(x))$.  
Then the bundle over $P(m,X,\sigma)$ associated to $(n\epsilon_\mathbb{R},\varepsilon_{k,l})$ is isomorphic to 
$k\xi \oplus l\epsilon_{\mathbb{R}}$. 
When $n=2d, k=l=d$, $n\epsilon_\mathbb{R}=\rho(d\epsilon_\mathbb{C})$ then the standard conjugation on 
$d\epsilon_\mathbb{C}$ equals 
$\varepsilon_{d,d}$ (for an obvious choice of $s_j, 1\le j\le n$). 

Let $(\omega,\hat \sigma)$ be a $\sigma$-conjugate complex vector bundle and let $\eta$ be a real vector bundle 
which is isomorphic to the real vector bundle $\rho(\omega)$ underlying $\omega$.    
Suppose that $f:\rho(\omega)\to \eta$ is a bundle isomorphism that covers the identity map of $X$. 
Set $\tilde\sigma:=f\circ \hat\sigma \circ f^{-1}$.   Then $\wt{\sigma}$ is an involution of $\eta$ that 
covers $\sigma$ and hence defines a vector bundle $\hat{\eta}$ over $P(m,X,\sigma)$.

\begin{lemma}  \label{realisomorphism}  We keep the above notations.
(i) The real vector bundles $\hat\omega$ and 
$\hat\eta$ over $P(m,X,\sigma)$ associated to the pairs $(\omega, \hat \sigma)$ and $(\eta,\tilde \sigma)$ are isomorphic.  
In particular $\hat\omega\cong \hat{\bar\omega}$.\\
(ii) Suppose that $\rho(\omega)=\eta_0\oplus \eta_1$ where $\eta_j, j=0,1$ are real vector bundles.  
Suppose that $\hat\sigma(E(\eta_j))=E(\eta_j)$, then 
$\hat\omega$ is isomorphic to $\hat\eta_0\oplus \hat\eta_1$ where $\hat\eta_j$ is defined with respect to 
the pair $(\eta_j, \hat\sigma|_{E(\eta_j)})$, $j=0,1$.  \\
(iii)  Let $n=k+l\ge 1$.  
Suppose that $\rho(\omega)\oplus n\epsilon_\mathbb{R}\cong N\epsilon_\mathbb{R}$, where $N:=2d+n$, and 
that $\varepsilon_{d+k,d+l}$ on $N\epsilon_{\mathbb{R}}$ restricts to $\hat \sigma$ on $\rho(\omega)$
and to $\varepsilon_{k,l}$ on $n\epsilon_\mathbb{R}$. 
Then 
$\hat \omega\oplus k\xi\oplus l\epsilon_\mathbb{R} \cong (d+k)\xi\oplus (d+l) \epsilon_l$.
\end{lemma}
\begin{proof}  We will only prove (i); the proofs of remaining parts are likewise straightforward.  
Consider the map $\phi: \mathbb{S}^m\times E(\omega)\to \mathbb{S}^m\times E(\eta) $ defined as $\phi(v, e)
=(v, f(e))~\forall v\in \mathbb{S}^m,e\in E(\omega)$.  
The $\phi((-v, \sigma (e)))=(-v, f(\hat \sigma(e)))=(-v, \tilde{\sigma}(f(e)))$.  Thus 
$\phi$ is $\mathbb{Z}_2$-equivariant and so induces a vector bundle homomorphism 
$\bar\phi: P(m,E(\omega), \hat\sigma)\to P(m, E(\eta),\tilde\sigma)$ that covers the identity map of 
$P(m,X,\sigma)$.  Restricted to each fibre, the map $\bar\phi$ is an $\mathbb{R}$-linear isomorphism 
since this is true of $f$.  
Therefore $\hat\omega$ and $\hat \eta$ are isomorphic vector bundles.  
Finally, let $\eta=\bar\omega, \tilde\sigma=\hat\sigma$ and $f=id$. Then 
$\hat\omega\cong \hat{\bar{\omega}}$.  
\end{proof}

\begin{example}\label{2sphere}
{\em 
(i)  Consider the Riemann sphere $\mathbb{S}^2=\mathbb{C}P^1$.  Let $\gamma\subset 
2\epsilon_\mathbb{C}$ be the tautological (complex) line bundle over $\mathbb{C}P^1$ and let $\beta$ be 
its orthogonal complement. 
As complex line bundles one has the isomorphism $\beta\cong \bar\gamma$.  
It follows that from the above lemma that $\hat \gamma\cong \hat \beta$.  Also $2\hat\gamma\cong \hat\gamma\oplus \hat\beta
\cong2\hat\epsilon_\mathbb C\cong 2\xi\oplus 2\epsilon_\mathbb R$.\\
(ii)  Suppose that $X=\mathbb{C}G_{n,k}$ and let $\sigma:X\to X$ be the conjugation $L\to \bar L$.   As seen 
in Example \ref{basic}(ii), $v\mapsto \bar v$ define conjugations of $\gamma_{n,k}, \beta_{n,k}$ that cover $\sigma$.  
Note that $\gamma_{n,k}\oplus \beta_{n,k}=n\epsilon_\mathbb{C}$.  By the above lemma we obtain 
that $\hat\gamma_{n,k}\oplus \hat \beta_{n,k}\cong d\hat{\epsilon}_\mathbb{C}\cong 
d\epsilon_\mathbb{R}\oplus d\xi$.   Also, the conjugations on $\gamma_{n,k},\beta_{n,k}$ induce an 
involution, denoted $\hat\sigma$, on $\omega:=\hom(\gamma_{n,k},\beta_{n,k})$; see Example \ref{basic}(iv).  
One has the isomorphism $\tau \mathbb{C}G_{n,k}\cong \omega$ of complex vector bundles (\cite{lam}). Under this isomorphism, the bundle involution $\hat\sigma$ corresponds to $T\sigma:T\mathbb{C}G_{n,k}\to T\mathbb{C}G_{n,k}$.  Therefore  $\hat\omega\cong \hat \tau \mathbb{C}G_{n,k}$. 
}
\end{example}

\subsection{Splitting principle}\label{split}
Denote by $\flag(\mathbb{C}^r)$ the complete flag manifold $\mathbb{C}G(1,\ldots,1)$. 
Let $\omega$ be a complex vector bundle over $X$ of rank $r\ge 1$ endowed with a Hermitian 
metric and let $q:\flag(\omega)\to X$ be 
the $\flag(\mathbb{C}^r)$-bundle associated to $\omega$.  Thus the fibre over an $x\in X$ is the 
space  $\{(L_1,\ldots,L_r)\mid L_1+\cdots+L_r=p_\omega^{-1}(x), L_j\perp L_k, 1\le j<k\le r, \dim_\mathbb{C}L_j=1
\}\cong\flag(\mathbb{C}^r)$ of complete flags in $p_\omega^{-1}(x)\subset E(\omega)$.  The vector bundle $q^*(\omega)$ splits 
as a Whitney sum $q^*(\omega)=\oplus_{1\le j\le r} \omega_j$ of complex line bundles $\omega_j$ over 
$\flag(\omega)$ with projection $p_j:E(\omega_j)\to \flag(\omega)$.   The fibre over a point $\mathbf{L}=(L_1,\ldots, L_r)\in 
\flag(\omega)$ of the bundle $\omega_j$ is the vector space $L_j\subset p_\omega^{-1}(q(\mathbf{L}))$.  

Suppose that $\sigma:X\to X$ is an involution and that $\hat{\sigma}:E(\omega)\to E(\omega)$ is a $\sigma$-conjugation on $\omega$.  
We shall write $\bar{e} $ for $\hat{\sigma}(e) , e\in E(\omega)$.    
One has the involution $\theta: \flag(\omega)\to \flag(\omega)$ 
defined as $\mathbf{L}=(L_1,\ldots, L_r)\mapsto (\bar{L}_1,\ldots,\bar{L}_r)=:\bar{\mathbf{L}}$.
Here $\bar{V}$ denotes the subspace $\hat{\sigma}(V)\subset p^{-1}_\omega(\sigma(x))$ when 
$V\subset p^{-1}_\omega(x).$   Then $\hat\theta:E(q^*(\omega))\to E(q^*(\omega))$ defined as 
$\hat\theta(\mathbf{L}, e)
=(\bar{\mathbf{L}}, \bar e)$ is a $\theta$-conjugation on $q^*(\omega)$. Moreover, it restricts 
to a $\theta$-conjugation $\hat\theta_j$ on the subbundle $\omega_j$ for each $j\le r$.  

Recall from \S\ref{assocbundle} that $\hat\omega$ is the real vector bundle with projection 
$p_{\hat\omega} :P(m,E(\omega),\hat{\omega})\to P(m,X,\sigma)$.  
Likewise, we have the real $2$-plane 
bundle $\hat{\omega}_j$ over $P(m,\flag(\omega),\theta)$ with projection $p_{\hat{\omega}_j}: P(m,E(\omega_j), \hat{\theta}_j)
\to P(m,\flag(\omega), \theta)$.     
Since $q\circ \theta=\sigma\circ q$, we 
have the induced map $\hat{q}: P(m,\flag(\omega),\theta)
\to P(m,X,\sigma)$ defined as $[v, \mathbf{L}]\mapsto  [v, q(\mathbf{L})]$.   The map $\hat{q}$ is in fact 
the projection of a fibre bundle with fibre the flag manifold $\flag(\mathbb{C}^r)$.    Since $\hat\theta=(\hat\theta_1,\ldots,\hat\theta_r)$, 
applying Lemma 
\ref{realisomorphism} (ii) 
we see that $\hat{q}^*(\hat\omega)\cong \oplus_{1\le j\le r}\hat\omega_j$. 

Recall that the first Chern classes mod $2$ of the canonical 
complex line bundles $\xi_j$ over $\flag(\mathbb{C}^r)$, $1\le j\le r$, generate the $\mathbb{Z}_2$-cohomology algebra 
$H^*(\flag(\mathbb{C}^r);\mathbb{Z}_2)$.   In fact $H^*(\flag(\mathbb{C}^r);\mathbb{Z})\cong  \mathbb{Z}[c_1,
\ldots, c_r]/I$ where $I$ is the ideal generated by the elementary symmetric polynomials in $c_1,\ldots,c_r$.  
Here the generators $ c_j+I$  may be identified with the (integral) Chern class $c_1(\xi_j)$.  
In particular 
$H^*(\flag(\mathbb{C}^r);\mathbb{Z})^{S_r}=H^0(\flag(\mathbb{C}^r);\mathbb{Z})\cong \mathbb{Z}$.  The last assertion 
is not valid for mod $2$-cohomology.  Indeed, the top dimensional $\mod 2$-cohomology group, being isomorphic to $ 
\mathbb Z_2$, is also fixed by $S_r$.

Since $\hat{\omega}_j$ restricts to the (real) $2$-plane bundle 
$\rho(\xi_j)$, we have $c_1(\xi_j)=i^*(w_2(\omega_j))$ where $i:\flag(\mathbb{C}^r)\cong \hat q^{-1}([v,x])\to P(m,\flag(\omega),\theta)$ is fibre 
inclusion, we see that the  
$\flag(\mathbb{C}^r)$-bundle $(P(m, \flag(\omega), \theta), P(m,X,\sigma), \hat{q})$ admits a $\mathbb{Z}_2$-cohomology 
extension of the fibre. 
By Leray-Hirsch theorem \cite[\S7, Ch.V]{spanier}, we have $H^*(P(m,\flag(\omega),\theta);\mathbb{Z}_2)\cong H^*(P(m,X,\sigma);\mathbb{Z}_2)
\otimes H^*(\flag(\mathbb{C}^r);\mathbb{Z}_2)$.  Thus $H^*(P(m,\flag(\omega),\theta);\mathbb{Z}_2)$ is a free module over the algebra $H^*(P(m,X,\sigma);\mathbb{Z}_2)$ of rank $\dim_{\mathbb{Z}_2} H^*(\flag(\mathbb{C}^r);\mathbb{Z}_2)=r!$.  In particular, it follows that 
$\hat{q}$ induces a monomorphism in mod $2$ cohomology.    

The symmetric group $S_r$ operates on $\flag(\omega)$ by permuting the components of each 
flag $\mathbf{L}=(L_1,\ldots,L_r)$ and the projection $q: \flag(\omega)\to X$  is constant on 
the $S_r$-orbits. Moreover,  $\theta\circ \lambda=\lambda\circ \theta$ for each $\lambda\in S_r$.  This implies that 
the $S_r$ action on $\flag(\omega)$ extends to an action on $P(m,\flag(\omega), \theta)$ where $\lambda([v,\mathbf{L}])
=[v, \lambda(\mathbf{L})]$.  The projection $\hat{q}: P(m,\flag(\omega), \theta)\to P(m,X,\sigma)$ 
is constant on $S_r$-orbits.  It follows that the image of the ring homomorphism 
$\hat{q}^*:H^*(P(m,X,\sigma);\mathbb{Z}_2)\to H^*(P(m,\flag(\omega),\theta);\mathbb{Z}_2)$ is contained in the 
subring $H^*(P(m,\flag(\omega),\theta);\mathbb{Z}_2)^{S_r}$ of elements fixed by the induced 
action of $S_r$ on $H^*(P(m,\flag(\omega),\theta);\mathbb{Z}_2)$.   As the $S_r$-action induces the identity map 
of $P(m, X,\sigma)$ 
we see that it acts as $H^*(P(m,X,\sigma);\mathbb{Z}_2)$-module automorphisms on 
$H^*(P(m,\flag(\omega),\theta);\mathbb{Z}_2)$.   
Hence $Im(\hat q^*)$ is contained in the subalgebra of $H^*(P(m,\flag(\omega),\theta),\mathbb Z_2)$ invariant under the 
action of $S_r$.

We summarise the above discussion in the proposition below.  

\begin{proposition} \label{splitting} {\em (Splitting principle) }
Let $\omega$ be a $\sigma$-conjugate complex vector bundle of rank $r$ and let $q:\flag(\omega)\to X$ be 
the associated $\flag(\mathbb{C}^r)$-bundle over $X$.  Then, with the above notations, \\(i) 
 the $\omega_j$ are $\theta$-conjugate line bundles for $1\le j\le r$, and, 
$ \hat{q}^*(\hat\omega)=\oplus_{1\le j\le r}\hat\omega_j.$\\
(ii)  $\hat{q}:P(m,\flag(\omega),\theta)\to P(m,X,\sigma)$ 
induces a monomorphism in cohomology, moreover, $H^*(P(m,\flag(\omega),\theta);\mathbb{Z}_2)$ is isomorphic, as an $H^*(P(m,X,\sigma);\mathbb{Z}_2)$-module, to a free module with basis a $\mathbb{Z}_2$-basis 
of $H^*(\flag(\mathbb{C}^r);\mathbb{Z}_2)$. \\
 (iii)  The image of $\hat{q}^*$ is contained in the subalgebra 
invariant under the action of the symmetric group 
$S_r$ on $H^*(P(m,\flag(\omega),\theta);\mathbb{Z}_2)$. \hfill $\Box$
\end{proposition}

We end this section with the following lemma which will be used in the sequel.

\begin{lemma} 
We keep the above notations.  Let $\omega$ be a $\sigma$-conjugate complex vector bundle 
over $X$.  Suppose that $\fix(\sigma)\ne \emptyset$ and that $H^1(X;\mathbb{Z}_2)=0$.   Then 
$\fix(\theta)\ne \emptyset$ and 
$ H^1(P(m,\flag(\omega),\theta);\mathbb{Z}_2)\cong H^1(P(m,X,\sigma);\mathbb{Z}_2)\cong H^1(\mathbb{R}P^m;\mathbb{Z}_2)\cong\mathbb{Z}_2$.
\end{lemma}
\begin{proof} 
Let $\sigma(x)=x\in X$ and set $V:=p_\omega^{-1}(x)$.  Then $\hat{\sigma}$ restricts to a conjugate complex 
isomorphism $\hat{\sigma}_x$ of $V$ onto itself.  Thus $V\cong \bar{V}$.   Then, setting $\fix(\hat\sigma_x)=:U
\subset V$, we see that $V$ is the $\mathbb{C}$-linear extension of $U$, that is,  
$V=U\otimes_\mathbb{R} \mathbb{C}$.  The Hermitian product on $V$ restricts to a (real) 
inner product on $U$.   
Let $(K_1,\ldots,K_r)$ be a complete real flag in $U$ and define $L_j:= K_j\otimes_\mathbb{R}\mathbb{C}\subset V$.
Then it is readily seen that $\mathbf{L}=(L_1,\ldots, L_r)$ belongs to $\flag(\omega)$ and is fixed by $\theta$.  

Since $H^1(X;\mathbb{Z}_2)=0$, we have $H^1(P(m,X,\sigma);\mathbb{Z}_2)\cong H^1(\mathbb{R}P^m;\mathbb{Z}_2)\cong \mathbb{Z}_2$, using the Serre spectral sequence of the $X$-bundle with projection $\pi:P(m,X,\sigma)\to \mathbb{R}P^m$.    The same argument applied to the 
$\flag(\mathbb{C}^r)$-bundle with projection $q:\flag(\omega)\to X$   
yields that $H^1(\flag(\omega);\mathbb{Z}_2)\cong H^1(X;\mathbb{Z}_2)=0$.  Now 
using the $\flag(\omega)$-bundle with projection $\hat{q}: P(m,\flag(\omega),\theta)\to P(m,X,\sigma)$, 
we obtain that 
 $H^1(P(m,\flag(\omega),\theta); \mathbb{Z}_2)\cong 
H^1(P(m,X,\sigma);\mathbb{Z}_2)\cong \mathbb{Z}_2$.   
\end{proof}

We shall identify $H^1(P(m,\flag(\omega),\theta);\mathbb{Z}_2), 
H^1(P(m,X,\sigma);\mathbb{Z}_2), H^1(\mathbb{R}P^m;\mathbb{Z}_2)$ and denote the generator of any one of them by $x$. \footnote{This should however cause no confusion with the notation for a typical point of  $X$.}


\subsection{A formula for Stiefel-Whitney classes of $\hat{\omega}$}
 Denote the Stiefel-Whitney polynomial $\sum_{0\le i\le q}w_i(\eta)t^i$ of a rank $q$ real vector bundle 
 $\eta$ by 
 $w(\eta;t)$ and similarly the Chern polynomial $\sum_{0\le i\le q}c_j(\alpha) t^j$ of a complex vector bundle $\alpha$ of rank 
 $q$ by  $c(\alpha;t)$.  Recall that when $\alpha$ is regarded as a real vector bundle, we have $w(\alpha; t)
 =c(\alpha;t^2) \mod 2$.  (See \cite{ms}.)  
 
 {\it We shall make no notational distinction between $c_j(\alpha)\in H^{2j}(X;\mathbb{Z})$ and its 
 reduction mod $2$ in $H^{2j}(X;\mathbb{Z}_2)$.}   In fact, we will mostly be working with 
 $\mathbb{Z}_2$-coefficients.

Since $\hat{\omega}$ restricted to any fibre of $\pi:P(m,X,\sigma)\to \mathbb{R}P^m$ is 
isomorphic to $\omega$  (regarded as a real vector bundle), we obtain that, the total Stiefel-Whitney 
polynomial 
$j^*(w(\hat{\omega};t))=w(\omega;t)=c(\omega,t^2)$ where $j: X\to P(m,X,\sigma)$ is the 
fibre inclusion.  

The following proposition yields the Stiefel-Whitney classes of $\hat{\omega}$ when $\omega$ is a complex 
line bundle.  Using this and the splitting principle, we will obtain a formula for the Stiefel-Whitney classes 
when $\omega$ is of arbitrary rank.  The proposition was obtained in the special case of Dold manifolds in  
\cite[Prop. 1.4]{ucci}.   Recall that $\xi$ is the line bundle associated to the double cover 
$\mathbb{S}^m\times X
\to P(m,X,\sigma)$ and is isomorphic to $\pi^*(\zeta)$.

\begin{lemma}  \label{hatomega}
Let $\sigma:X\to X$ be an involution with non-empty fixed point set and 
let $\omega$ be a complex vector bundle of rank $r$ over $X$. 
With the above notations, we have 
 $\hat{\omega}\cong \xi\otimes \hat{\omega}$.
\end{lemma} 
\begin{proof}  
 The total space of the bundle $\xi\otimes \hat{\omega}$ has the description $E(\xi\otimes\hat{\omega})
=\{[v,x; t\otimes e]\mid [v,x]\in P(m,X;\sigma),t\in \mathbb{R}, e\in p_\omega^{-1}(x)\}$ where $[v,x;t\otimes e]
=\{(v,x;t\otimes e), (-v,\sigma(x); -t\otimes \hat{\sigma}(e))\}$; here $\hat{\sigma}:E(\omega)\to E(\omega)$ is an {\it involutive}  
bundle map that covers $\sigma$ and is conjugate linear isomorphism on each fibre.  Thus 
we have the equality $\hat{\sigma}(\sqrt{-1}te)=-\sqrt{-1}t\hat{\sigma}(e)$.  
Observe that  $ [v,x; \sqrt{-1}te]=[-v,\sigma(x); \hat{\sigma}(\sqrt{-1}te)]=[-v,\sigma(x), -\sqrt{-1}t\hat{\sigma}(e)]$ 
and so the map $h:E(\xi\otimes \hat{\omega})\to E(\hat{\omega})$,  
$[v,x;t\otimes e]\mapsto [v,x; \sqrt{-1}te]=[-v,\sigma(x); -\sqrt{-1}t\hat{\sigma}(e)]$ is a well-defined isomorphism of real vector bundles.   
\end{proof}

\subsection*{Simplifying assumptions} \label{assumptions}  We shall make the following simplifying assumptions.\\
(a) $\sigma:X\to X$ has a fixed point.  As observed already, the $X$-bundle $\pi:P(m,X,\sigma)\to 
\mathbb{R}P^m$ admits a cross-section $s:\mathbb{R}P^m\to P(m,X,\sigma)$.    
It follows that  $\pi^*: H^*(\mathbb{R}P^m;\mathbb{Z}_2)
\to H^*(P(m,X,\sigma);\mathbb{Z}_2)$ is a monomorphism.  We shall 
identify $H^*(\mathbb{R}P^m;\mathbb{Z}_2)$ with its image under $\pi^*$. \\
(b) $H^1(X;\mathbb{Z}_2)=0$.  This implies that $H^2(X;\mathbb{Z})\to H^2(X;\mathbb{Z}_2)$ induced by the 
homomorphism $\mathbb{Z}\to \mathbb{Z}_2$ of the coefficient rings is surjective.  \\

\begin{example}\label{flag-assumptions}{\em 
(i)  Let $X$ be the complex flag manifold $\mathbb{C}G(n_1,\ldots,n_r)$ and let 
$\sigma:X\to X$ be defined by the complex conjugation on $\mathbb{C}^n$, $n=\sum n_j$.  Then $\fix(\sigma)$ is the {\it real} flag manifold $\mathbb{R}G(n_1,\ldots, n_r)=O(n)/(O(n_1)\times\cdots\times O(n_r))$ so assumption (a) holds.  Since $X$ is simply connected, (b) also holds. 

(ii)  Let $\omega$ be a $\sigma$-conjugate complex vector bundle of rank $r$.
Suppose that $\fix(\sigma)\ne \emptyset$ and that $H^1(X;\mathbb{Z}_2)=0$.   
Let $\theta:\flag(\omega)\to \flag(\omega)$ 
be the associated involution of the $\flag(\mathbb{C}^r)$-manifold bundle over $X$. (See \S\ref{split}.)  
Then $\fix(\theta)\ne \emptyset$ and 
$H^1(\flag(\omega);\mathbb{Z}_2)=0$.  
}
\end{example}

In the Serre spectral sequence of the bundle $(P(m,X), \mathbb{R}P^m,X,\pi)$, 
we have $E_2^{0,k}=H^0(\mathbb{R}P^m;\mathcal{H}^k(X;\mathbb{Z}_2))$ where $\mathcal{H}^k(X;\mathbb{Z}_2)$ denotes the 
local coefficient system on $\mathbb{R}P^m$.  
The action of the fundamental group of $\mathbb{R}P^m$ on $H^*(X;\mathbb{Z}_2)$ 
is generated by the involution  $\sigma^*: H^*(X;\mathbb{Z}_2)\to H^*(X;\mathbb{Z}_2)$.  
Hence $E^{0,2}_2=H^2(X;\mathbb{Z}_2)^{\mathbb{Z}_2}=\fix(\sigma^*)$.   In order to emphasise the dimension, we shall write $H^2(\sigma;\mathbb{Z}_2)$ instead of 
$\sigma^*$.  
Also (b) implies that $E_3^{0,2}=E^{0,2}_2$ and (a) implies that the transgression $E_3^{0,2}=\fix(H^2(\sigma;\mathbb{Z}_2))\to E_3^{3,0}
=H^3(\mathbb{R}P^3;\mathbb{Z}_2)$ is zero.  It follows that $E_3^{0,2}=E_\infty^{0,2}$ and that the image 
$j^*:H^2(P(m,X);\mathbb{Z}_2) \to H^2(X;\mathbb{Z}_2)$ equals $\fix(H^2(\sigma;\mathbb{Z}_2))$, where $j:X\hookrightarrow P(m,X)$ is the fibre inclusion.   
We have the exact sequence:
\[0\to H^2(\mathbb{R}P^m;\mathbb{Z}_2)\stackrel{\pi^*}{\to} H^2(P(m,X,\sigma);\mathbb{Z}_2)\stackrel{j^*}{\to} \fix(H^2(\sigma;\mathbb{Z}_2))\to 0. 
\eqno(1)\]

The homomorphism $s^*:H^2(P(m,X,\sigma);\mathbb{Z}_2)\to H^2(\mathbb{R}P^m;\mathbb{Z}_2)$ yields a splitting and allows us to identify $\fix(H^2(\sigma;\mathbb{Z}_2))$ as a {\em subspace} 
of $H^2(P(m,X,\sigma);\mathbb{Z}_2)$, namely the kernel of $s^*$.  {\it We shall denote the image of an element 
$u\in \fix(H^2(\sigma;\mathbb{Z}_2))$ by $\tilde{u}$.}

\begin{lemma}  \label{linebundle} 
Suppose that $\sigma(x_0)=x_0$ and $H^1(X;\mathbb{Z}_2)=0$.   
Let $s:\mathbb RP^m\to P(m,X,\sigma)$ be defined as $v\mapsto [v,x_0]$ and let  
$\omega$ be a $\sigma$-conjugate complex vector bundle over $X$ of rank $r$.  Then 
(i) $s^*(\hat\omega)\cong r\epsilon_\mathbb{R}\oplus r\zeta$, (ii) $c_k(\omega)\in \fix(H^{2k}(\sigma;\mathbb{Z}_2)), ~k\le r$, and, (iii) if $r=1$, then $w(\hat\omega)=1+x+\tilde{c}_1(\omega).$ 
\end{lemma}  
\begin{proof} (i) Since $\sigma(x_0)=x_0$, $\hat\sigma$ restricts to a conjugate complex linear automorphism $\hat\sigma_0$ of $V:=p_\omega^{-1}(x_0)$.   
Let $U\subset V$ is the eigenspace of $\hat\sigma_0$ corresponding to 
eigenvalue $1$ of $\hat\sigma_0$.   Then $\sqrt{-1}U$ is the $-1$ eigenspace.  The vector bundle 
$s^*(\hat\omega)$ is isomorphic to the Whitney sum of the bundles $\mathbb{S}^m\times_{\mathbb{Z}_2} U\to 
\mathbb{R}P^m$ and $\mathbb{S}^m\times_{\mathbb{Z}_2}\sqrt{-1}U\to \mathbb{R}P^m$.
Evidently these bundles are isomorphic to $r\epsilon_\mathbb{R}$ and $r\xi$ respectively.  

(ii)  Since $\hat\sigma: E(\omega)\to E(\omega)$ is a {\it conjugate} complex linear bundle map covering $\sigma$, we have 
$\sigma^*(\omega)\cong \bar \omega$.  So $\sigma^*(c_k(\omega))=c_k(\sigma^*(\omega))=(c_k(\bar\omega))=
(-1)^kc_k(\omega)\in H^{2k}(X;\mathbb{Z})$.  Therefore $c_k(\omega)\in \fix(H^{2k}(\sigma;\mathbb{Z}_2)), ~k\le r.$

(iii)  Using the isomorphism $s^*:H^1(P(m,X);\mathbb{Z}_2)\cong H^1(\mathbb{R}P^m;\mathbb{Z}_2)$, it follows from (i) that  $w_1(\hat\omega)=w_1(\xi)=x$.  
Since $c_1(\omega)\in \fix(H^2(\sigma;\mathbb{Z}_2))$, the element $\tilde{c}_1(\omega)$ is meaningful.  
It remains to show that $w_2(\hat\omega)=\tilde c_1(\omega)$.  Since $j^*(\hat\omega)=\omega$, we 
see that $j^*(w_2(\hat\omega))=w_2(\omega)=c_1(\omega)\in \fix(H^2(\sigma;\mathbb{Z}_2))$.  On the other hand, $w_2(s^*(\hat\omega))=0$. 
So, under our identification of $\fix(H^2(\sigma;\mathbb{Z}_2))$ with the kernel of $s^*$, we have $w_2(\hat\omega)=
\tilde c_1(\omega)$.
\end{proof}

\begin{remark}
{\em 
The above lemma shows that the element 
$\tilde{c}_1(\omega)\in H^2(P(m,X);\mathbb{Z}_2)$ is independent of the choice of the fixed point $x_0\in X$ 
(used in the definition of $s^*$) since it equals $w_2(\hat{\omega})$.
}
\end{remark}

Suppose that $\omega$ is a $\sigma$-conjugate complex vector bundle of rank $r$ over $X$.   
Since $q^*(\omega)$ splits as a Whitney sum $q^*(\omega)=\oplus_{1\le j\le r} \omega_j$, where $q:\flag(\omega)\to X$ is the $\flag(\mathbb{C}^r)$-bundle, in view of Example \ref{flag-assumptions}, we have $c_1(\omega_j)\in \fix(H^2(\theta;\mathbb{Z}_2))$. 
Therefore we obtain their `lifts' $\tilde{c}_1(\omega_j)\in H^2(P(m,\flag(\omega);\theta);\mathbb{Z}_2)$.     
The bundle $\hat{q}^*(\hat\omega)$ splits as $\hat q^*(\hat\omega)=\oplus_{1\le j\le r}\hat\omega_j$ (see Proposition \ref{splitting}(i)), 
where $\hat{q}: P(m,\flag(\omega),\theta)\to P(m,X,\sigma)$ is the projection of the $\flag(\mathbb{C}^r)$-bundle.
Therefore $e_j(\tilde{c}_1(\omega_{1}),\ldots, \tilde{c}_1(\omega_{r}))=
e_j(w_2(\hat\omega_1),\ldots, w_2(\hat\omega_r))$ 
is in $H^{2j}(P(m,X,\sigma);\mathbb{Z}_2)$. 
Here $e_j$ stands for the $j$-th elementary symmetric polynomial.

\noindent
{\bf Notation:} {\it 
Set $\tilde{c}_j(\omega):=e_j(w_2(\hat\omega_1), \ldots, w_2(\hat\omega_r))\in 
H^{2j}(P(m,X,\sigma);\mathbb{Z}_2), ~1\le j\le r.$}

When $j>r$, $\tilde{c}_j=0$.  
Observe that $\tilde{c}_j(\omega)$ restricts to $c_j(\omega)
\in H^{2j}(X;\mathbb{Z}_2)$ on any fibre of $\pi: P(m,X,\sigma);\mathbb{Z}_2)\to \mathbb{R}P^m$.   

We have the following formula for the Stiefel-Whitney classes of $\hat\omega$. 

\begin{proposition}  \label{swformula1}
We keep the above notations.  Let 
$\omega$ be a $\sigma$-conjugate complex vector bundle over $X$.  Suppose that $H^1(X;\mathbb{Z}_2)=0$ 
and that $\fix(\sigma)\ne \emptyset$. 
Then, 
\[w(\hat\omega;t)=\sum_{0\le j\le r}(1+xt)^{r-j} \tilde{c}_{j}(\omega)t^{2j}.\eqno(2)\]
\end{proposition}
\begin{proof}
The case when $\omega$ is a line bundle was settled in Lemma \ref{linebundle}.
In the more general case, we apply 
the splitting principle, Proposition \ref{splitting}(i).
The bundle isomorphism 
$\hat q^*(\hat\omega)=\hat\omega_1\oplus\cdots\hat\omega_r$ given in Proposition \ref{splitting}(i) 
leads to the formula 
\[w(\hat\omega;t)=\prod_{1\le j\le r}(1+xt+\tilde{c}_1(\omega_j)t^2). \]
The proposition follows from Lemma \ref{linebundle} and the definition of $\tilde{c}_j(\omega)$ 
since $w_2(\hat\omega_j)=\tilde c_1(\omega_j)$. 
\end{proof}

\section{The tangent bundle of $P(m,X)$}
Let $X$ be a connected almost complex manifold and let $\sigma:X\to X$ be a complex conjugation. 
Thus $\hat\sigma=T\sigma$ is a $\sigma$-conjugation.  The manifold $P(m,X,\sigma)$ will be more briefly denoted $P(m,X)$.  
The bundle $\hat\tau X$ restricts to the tangent bundle along 
any fibre of $\pi: P(m,X)\to \mathbb{R}P^m$ and so is a subbundle of $\tau P(m,X)$.  Clearly $\hat\tau X$ is contained in the  
kernel of $T\pi:TP(m,X)\to T\mathbb{R}P^m$.  In fact $\hat \tau X=\ker(T\pi)$ since their ranks are equal.  Therefore we have 
a Whitney sum decomposition 
\[\tau P(m,X)=\pi^*(\tau \mathbb{R}P^m)
\oplus \hat \tau X. \eqno(3)\]   
We assume that 
$\fix(\sigma)$ is non-empty and hence a smooth manifold of dimension $d=(1/2) \dim X$.  Also we assume 
that $H^1(X;\mathbb{Z}_2)=0$. 
Using the fact that $w(\mathbb{R}P^m)=(1+x)^{m+1}$, and applying 
Proposition \ref{swformula1}, we have 

\begin{theorem} \label{swformulaforp} 
Let $X$ be a connected compact almost complex manifold with complex conjugation 
$\sigma$.  Suppose that $\fix(\sigma)\ne \emptyset$ and that $H^1(X;\mathbb{Z}_2)=0$.
Then: \[w(P(m,X);t)=(1+xt)^{m+1} .\sum_{0\le j\le d}(1+xt)^{d-j}\tilde{c}_j(X)t^{2j}.
\eqno(4)~~~\Box\] 
\end{theorem}

As an application of the above theorem we obtain 

\begin{corollary} \label{spin} 
(i) $P(m,X)$ is orientable if and only if $m+d$ is odd.\\
 (ii) $P(m,X)$ admits a spin structure if and only if $X$ admits a spin structure and $m+1\equiv d \mod 4$ when $m>1$. 
\end{corollary}
\begin{proof}   
Since $P(m,X)=(\mathbb{S}^m\times X)/\mathbb{Z}_2$, it is readily seen 
that $P(m,X)$ is orientable if and only if the antipodal map of $\mathbb{S}^m$ 
and the conjugation involution $\sigma$ on $X$ are simultaneously either orientation 
preserving or orientation reversing. The latter condition is equivalent to $m+1\equiv d\mod 2$.
Alternatively, from Theorem \ref{swformulaforp}, we obtain that  
$w_1(P(m,X))=(m+1+d)x$, which is zero precisely if $m+d$ is odd.  

Using the same formula, we have 
$w_2(P(m,X))=({m+1\choose 2}+{d\choose 2})x^2 +\tilde{c}_1(X)$.   The existence of a spin 
structure being equivalent to vanishing of the first and the second Stiefel-Whitney classes, 
we see that $P(m,X)$ admits a spin structure if and only if $X$ admits a spin structure    
and ${m+1\choose 2}\equiv {d\choose 2} \mod 2$ with $m+d$ odd.  The latter condition is 
equivalent to $m+1\equiv d\mod 4$.
\end{proof}

The notions 
of stable parallelizability and parallelizability were recalled in the Introduction.   
Recall from \S\ref{dependence} the $\sigma$-conjugation $\varepsilon_{k,n-k}:X\times \mathbb{R}^n\to X\times \mathbb{R}^n$, defined with respect to a set of everywhere linearly independent sections $s_1,\ldots,s_n$.

\begin{theorem} \label{stableparallel}  Let $\sigma$ be a conjugation on a connected almost complex manifold $X$ and let $\dim_\mathbb{R}X=2d$.   Suppose that $\fix(\sigma)\ne \emptyset$. Then:\\ 
(i)  If $P(m,X)$ is stably parallelizable, then 
$X$ is stably parallelizable and $2^{\varphi(m)}|(m+1+d)$.  \\  
(ii)  Suppose that $\rho(\tau X)\oplus n\epsilon_\mathbb{R}\cong (2d+n) \epsilon_\mathbb{R}$ as {\em real} vector bundle. Suppose that the bundle map $\varepsilon_{d+k,d+n-k}$ of $(2d+n)\epsilon_{\mathbb{R}}$ covering $\sigma$ 
restricts to $\hat\sigma=T\sigma$ on $TX$ and to $\varepsilon_{k,n-k}$ on $n\epsilon_\mathbb{R}$.
If $2^{\varphi(m)}|(m+1+d) $, then $P(m,X)$ is stably parallelizable. \\
(iii) Suppose that $m$ is even and that $P(m,X)$ is stably parallelizable.    
Then $P(m,X)$ is parallelizable if and only 
if $\chi(X)=0$.
\end{theorem}

\begin{proof}
(i)  If $E\to B$ is {\it any} smooth fibre bundle with fibre $X$, the normal bundle to the fibre inclusion $X
\hookrightarrow E$ is trivial.  So if $E$ is stably parallelizable, then so is $X$.  It follows that 
stable parallelizability of $P(m,X)$ implies that of $X$.

Let $x_0\in \fix(\sigma)$ and let $s:\mathbb{R}P^m\to P(m,X)$ be the corresponding cross-section defined 
as $[v]\mapsto [v,x_0]$.  In view of  Lemma \ref{linebundle} and the bundle isomorphism (3), we see that $s^*(\tau P(m,X))=s^*(\pi^* \tau \mathbb{R}P^m
\oplus \hat \tau X)=\tau \mathbb{R}P^m\oplus d\epsilon_\mathbb{R}\oplus d\zeta\cong (m+1+d)\zeta\oplus (d-1)\epsilon_\mathbb{R}$.  Thus the stable parallelizability of $P(m,X)$ implies 
that $(m+1+d)([\zeta]-1)=0$ in $KO(\mathbb{R}P^m)$.  By the result of Adams \cite{adams} (recalled 
in \S\ref{intro})
it follows that $2^{\varphi(m)}|(m+1+d)$.   

(ii) Our hypothesis 
implies, using Lemma \ref{realisomorphism}, that 
$\hat{\tau}X \oplus (k\xi\oplus(n-k)\epsilon_\mathbb{R}) \cong (d+n-k)\epsilon_\mathbb{R}\oplus (d+k)\xi.$ 
Therefore, using the isomorphism (3),  
$\tau P(m,X)\oplus k\xi\oplus (n-k+1)\epsilon_\mathbb{R}\cong k\xi\oplus (n-k+1)\epsilon_\mathbb{R}\oplus 
\pi^*(\tau \mathbb{R}P^m)\oplus \hat{\tau} X\cong (m+1) \xi\oplus \hat{\tau}X\oplus k\xi\oplus (n-k)\epsilon_\mathbb{R}
\cong (m+1)\xi\oplus(d+k)\xi\oplus(d+n-k)\epsilon_\mathbb{R}$. 
Since $\dim P(m,X)=2d+m<2d+n+1+m,$ we may cancel the factor $k\xi\oplus (n-k)\epsilon_\mathbb{R}$ on both sides \cite[Theorem 1.1, Ch. 9]{husemoller}, leading to 
an isomorphism $\tau P(m,X)\oplus \epsilon_\mathbb{R} \cong (d+m+1) \xi\oplus d\epsilon_\mathbb{R}$.  Since $\xi=\pi^*(\zeta)$, again 
using Adams' result
 it follows that $P(m,X)$ is stably parallelizable if $2^{\varphi(m)}$ divides $(m+d+1)$.  

(iii) Since $m$ is even, $P(m,X)$ is even dimensional. 
By Bredon-Kosi\'nski's theorem \cite{bk}, it follows that $P(m,X)$ is parallelizable if and only if its span is at least 
$1$.  By Hopf's theorem, $\textrm{span~}P(m,X)\ge 1$ if and only if $\chi(P(m,X))$ vanishes.  Since $\chi(P(m,X))
=\chi(\mathbb{R}P^m).\chi(X)=\chi(X) $ as $m$ is even, the assertion follows.  
\end{proof}

The {\it stable span} 
of a smooth manifold $M$ is the largest number $s\ge 0$ such that $\tau M\oplus \epsilon_\mathbb{R}
\cong (s+1)\epsilon_\mathbb{R}\oplus \eta$ for some real vector bundle $\eta$.    We extend the notion of 
span and stable span to a (real) vector bundle $\gamma$ over a base space $B$ in an obvious mannner; thus 
$\textrm{span}(\alpha)$ is the largest number $r\ge 0$ so that $\gamma\cong\alpha\oplus r\epsilon_\mathbb{R}$ 
for some vector bundle $\alpha$.  If rank of $\gamma$ equals $n$ and if $B$ is a CW complex of dimension 
$d\le n$, then $\textrm{span}(\gamma)\ge n-d$.  See \cite[Theorem 1.1, Ch. 9]{husemoller}.   
It follows that if $n>d$, then $\textrm{span}(\gamma)
=\textrm{stable~span}(\gamma)$.

\begin{remark}{\em
(i) Suppose that $P(m,X)$ is stably parallelizable.  If $m$ is odd, then $\chi (P(m,X))=0$ as 
$\chi(\mathbb{R}P^m)=0$.  Consequently we obtain no information about $\chi(X)$ from the 
equality $\chi(P(m,X))=\chi(\mathbb{R}P^m) \chi(X)$.   Let us suppose that $\chi(X)\ne 0$.
Since $\textrm{span} ( \mathbb{R}P^m)=\textrm{span}(\mathbb{S}^m),$ 
we obtain the lower bound $\textrm{span} (P(m,X))\ge \textrm{span}(\mathbb{S}^m)=\rho(m+1)-1$, 
where $\rho(m+1)$ is the Hurwitz-Radon function defined as $\rho(2^{4a+b}(2c+1))=8a+2^b$, $0\le b<4, ~a,c \ge 0$. 
From Bredon-Kosi\'nski's theorem \cite{bk}, we obtain that $P(m,X)$ is parallelizable if $\rho(m+1)>\rho(m+2d+1). $ 
For example if $m=(2c+1)2^r-1$ and $d=2^{s}(2k+1)$ with $s< r-1$ then $m+1+2d=((2c+1)2^{r-1-s}+2k+1)2^{s+1}$ and 
so $\rho(m+1)=\rho(2^r)>\rho(2^{s+1})=\rho(m+2d+1)$; consequently $P(m,X)$ is parallelizable.    

(ii)  The following bounds for the span and stable span of $P(m,X)$ are easily obtained. \\
$\bullet ~ \textrm{stable ~span}(P(m,X))\le \min \{d
+\textrm{span}(m+d+1)\zeta, m+\textrm{stable~span}(X)\}$,\\ 
$\bullet~\textrm{span}(P(m,X))\ge \textrm{span}(\mathbb{R}P^m)$. \\
If $m$ is even and $\chi(X)=0$, then $\chi(P(m,X))=0.$  Since $\dim P(m,X)$ is even,  it follows 
by \cite[Theorem 20.1]{koschorke}, that $\textrm{span}(P(m,X))=\textrm{stable~span}(P(m,X))$. 
}   
\end{remark}

We illustrate Theorem \ref{stableparallel} in the case 
when $X$ is the complex flag manifold $\mathbb{C}G(n_1,\ldots, n_r)$, where the $n_j\ge 1$ are positive integers and 
$n=\sum_{1\le j\le r} n_j$, with its usual differentiable structure.   It admits an $U(n)$-invariant complex structure and the smooth involution $\sigma:X\to X$ defined by the complex conjugation on $\mathbb{C}^n$ 
is a conjugation, as remarked in Example \ref{flag-assumptions}(i).  
We assume, without 
loss of generality, that $n_1\ge \cdots\ge n_r$.   
We denote by $P(m; n_1,\ldots,n_r)$ the space 
$P(m,\mathbb{C}G(n_1,\ldots,n_r))$.  Note that $\mathbb{C}G(1,\ldots,1)$ is the complete flag manifold 
$\flag(\mathbb{C}^n)$.  

The classical Dold manifold corresponds to 
$r=2$ and $n_1\ge n_2=1$.  Theorem \ref{parallel-grassmann} in this special case is  
due to J. Korba\v{s} \cite{korbas}.   (Cf. \cite{ucci}, \cite{li}.)

\noindent
{\it Proof of Theorem \ref{parallel-grassmann}.  }
When $n_j>1$ for some $j$, the flag manifold $X=\mathbb{C}G(n_1,\ldots, n_r)$ is well-known to be 
{\it not} stably parallelizable; see, for example, \cite{sz}.  (Cf. \cite{korbas85}.)
So, by Theorem \ref{stableparallel}, the non-trivial part of theorem concerns the case when the 
flag manifold is stably parallelizable, namely, 
$n_j=1$ for all $j$.  It remains to determine the values of $m$ for which $P=P(m;1,\ldots,1)$ is 
stably parallelizable.  This is done in Proposition \ref{fullflag-parallel} below.

The manifold $X=\mathbb{C}G(1,\ldots,1)$ has non-vanishing Euler characteristic;   
in fact, $\chi(X)=n!$, the order of the Weyl group of $U(n)$.  When $m$ is even, it follows that 
$\chi(P)=n!$ and so  $\textrm{span}(P)=0$.     

Suppose that $\rho(m+1)>\rho(m+1+2{n\choose 2})$.  Then 
$\span (P)\ge \span(\mathbb{R}P^m)\ge \rho(m+1)-1$ whereas the span of the sphere 
of dimension $\dim P=m+2d=m+n(n-1)$ equals $\rho(m+1+n(n-1))-1$.  So, by Bredon-Kosi\'nski theorem \cite{bk}, 
$P$ is parallelizable if it is stably parallelizable and $\rho(m+1)>\rho(m+1+n(n-1))$. \hfill $\Box$ 

It is known that 
$\flag(\mathbb{C}^n)$ is stably parallelizable, but not parallelizable,  as a real manifold  (Cf. \cite[p.313]{lam}.)  
(The non-parallelizability of $\flag(\mathbb{C}^n)$ follows immediately from the fact that 
$\chi(\flag(\mathbb{C}^n))\ne 0$.) 

\subsection*{Lam's functor $\mu^2$}
As a preparation 
for the proof of Proposition \ref{fullflag-parallel} we recall a certain functor $\mu^2$ introduced by Lam \cite[\S\S4-5]{lam}.  This allows us to apply Lemma \ref{realisomorphism}(iii).

The functor $\mu^2=\mu^2_\mathbb{C}$ associates a real vector bundle to a complex vector bundle.\footnote{Lam defined $\mu^2$ in a more general setting that includes (left) vector bundles over quaternions as well.}   We assume the base space to be paracompact so that every complex vector bundle over it 
admits a Hermitian metric. 
If $V$ is any complex vector space $\mu^2(V)$ is defined as $\mu^2(V)=\bar V\otimes_\mathbb{C} V/\fix (\theta)$ where 
$\theta: \bar V\otimes V\to \bar V\otimes V$ is the conjugate complex linear automorphism 
defined as $\theta(u\otimes v)=-v\otimes u$.  As with any {\it continuous} functor (\cite[\S3(f)]{ms}), $\mu^2$ 
is determined by its restriction to the category of finite dimensional complex vector spaces and their isomorphisms.  
The functor $\mu^2$ has the following properties where $\omega, \omega_1,\omega_2$ are all complex vector bundles over a base space $X$.  The first three were established by Lam. \\
(i) $\textrm{rank} (\mu^2(\omega))=n^2$ where $n$ is the rank of $\omega$ as a complex vector bundle.  \\
(ii) $\mu^2(\omega)\cong \epsilon_\mathbb{R}$ if $\omega$ is a complex line bundle.  Indeed, choosing a 
positive Hermitian metric on $\omega$, 
the map $E(\mu^2(\omega))\ni [u\otimes zu]\mapsto (p_\omega(u), Re(z)||u||^2)\in X\times \mathbb{R}, z\in \mathbb{C}$ is a well-defined real vector bundle homomorphism. It is clearly non-zero and since the ranks agree, it is an isomorphism.\\
(iii)  $\mu^2(\omega_1\oplus \omega_2)=\mu^2(\omega_1)\oplus (\bar\omega_1\otimes_\mathbb{C}\omega_2)\oplus \mu^2(\omega_2)$.\\
(iv) If $\hat{\sigma}:E(\omega)\to E(\omega)$ is a conjugation of $\omega$ covering an involution $\sigma:X\to X$, then 
$\mu^2(\hat\sigma):E(\mu^2(\omega))\to E(\mu^2(\omega))$ is a bundle map covering $\sigma$.  In particular $\mu^2(\bar \omega)\cong \mu^2(\omega)$. \\
(v) If $\hat\sigma$ is a conjugation of a complex {\it line} bundle $\omega$ with a Hermitian metric $\langle .,.\rangle$ 
covering an involution $\sigma$ such that 
$\langle u,v\rangle_x=\overline{\langle \hat\sigma (u),\hat\sigma(v)\rangle}_{\sigma(x)}, u,v\in p^{-1}_\omega(x), x\in X$, 
then 
$\mu^2(\hat\sigma):\mu^2(\omega)\to \mu^2(\omega)$ is the identity on each fibre under the isomorphism 
$\mu^2(\omega)\cong \epsilon_\mathbb{R}$ of (ii) since $||\hat\sigma(u)||=||u||$.

 \begin{proposition}\label{fullflag-parallel}
 The manifold $P(m;1,\ldots,1)=P(m,\flag(\mathbb{C}^n))$ is stably parallelizable if and only if $2^{\varphi(m)}$ divides $(m+1+{n\choose 2})$.
 \end{proposition}
\begin{proof}
Recall (\cite[Corollary 1.2]{lam}) that $\tau \mathbb{C}G(n_1,\ldots, n_r)\cong \oplus_{1\le i<j\le r}
 \bar\gamma_i\otimes \gamma_j$
where $\gamma_j$ is the $j$-th canonical bundle of rank $n_j$ whose fibre over 
$(L_1,\ldots,L_r)\in \mathbb{C}G(n_1,\ldots, n_r)$ is the complex vector space $L_j$.  We have 
\[ \gamma_1\oplus \cdots\oplus\gamma_r\cong n\epsilon_\mathbb{C}.\]
Applying $\mu^2$ and using the above description of $\tau \mathbb{C}G(n_1,\ldots,n_r)$ we obtain the following 
isomorphism of {\it real} vector bundles by repeated use of property (iii) of $\mu^2$ listed above:
\[ \bigoplus \mu^2(\gamma_j)\oplus \tau (\mathbb{C}G(n_1,\ldots, n_r))\cong n\epsilon_\mathbb{R}\oplus (\bigoplus_{1\le i<j\le n}\epsilon_\mathbb{C}(\bar e_i\otimes e_j))\cong n^2\epsilon_\mathbb{R}.\eqno(5)\]   
(Cf. \cite[Theorem 5.1]{lam}.)
Specialising to the case of $X=\flag(\mathbb{C}^n)$ we have $\mu^2(\gamma_j)\cong \epsilon_\mathbb{R}$.
The involution $\sigma:X\to X$ defined as ${\bf L}\mapsto \bar{\bf L}$ induces a 
complex conjugation of $\hat \sigma=T\sigma$ on $\tau X$ which preserves the summands 
$\omega_{ij}:=\bar \gamma_i\otimes \gamma_j, i<j,$ yielding a conjugation $\hat\sigma_{ij}$ on it.  
The bundle involution $\varepsilon_{d,d}$ (covering $\sigma$) on the summand on the right $\oplus_{1\le i<j\le n} \rho(\epsilon_\mathbb{C})$, defined with respect to the basis $\bar e_i\otimes e_j, \bar e_i\otimes \sqrt{-1}e_j, 1\le i<j\le n,$ and $\varepsilon_{0, n}$ 
on the summand $\oplus_{1\le i\le n}\epsilon_\mathbb{R}(\bar e_i\otimes e_i)$ defined with respect to $\bar e_i\otimes e_i, 1\le i\le n,$ 
together define an involution, denoted $\varepsilon$, that covers $\sigma$.  Under the isomorphism, $\varepsilon$ restricts 
to $T\sigma$ on $\tau X$ and to $\varepsilon_{0,n}$ on $\oplus_{1\le i\le n} \mu^2(\gamma_i)$ defined with respect to 
a basis $\bar{u}_i\otimes u_i,1\le i\le n,$ where $u_i\in L_i$ with $||u_i||=1$.     
It follows, by using (v) above and Lemma \ref{realisomorphism}, that 
\[n\epsilon_\mathbb{R}\oplus \hat \tau \flag(\mathbb{C}^n)\cong n\epsilon_\mathbb{R}\oplus {n\choose 2}(\epsilon_\mathbb{R}\oplus \xi).\]
Therefore $(n+1)\epsilon_\mathbb{R}\oplus \tau P\cong (m+1)\xi\oplus \hat\tau \flag(\mathbb{C}^n)\oplus n\epsilon_\mathbb{R}\cong (m+1+{n\choose 2})\xi\oplus {n+1\choose 2}\epsilon_\mathbb{R}$.  Hence $\tau P$ is stably trivial 
if and only if $(m+1+{n\choose 2})\xi$ is stably trivial if and only if $(m+1+{n\choose 2})\zeta$ on $\mathbb{R}P^m$ 
is stably trivial if and only if $2^{\varphi(m)}$ divides $(m+1+{n\choose 2})$.  This completes the proof. 
\end{proof}  

\begin{remark}{\em
It is clear that for a given $n\ge 2$, there are only finitely many values $m\ge 1$ for which $P=P(m,\flag(\mathbb{C}^n))$ 
is parallelizable.  In fact, since $2^{\varphi(m)}\ge 2m$ for $m\ge 8$, we must have $m\le \max\{8,{n\choose 2}\}$.  
However the required values of $m$ are highly restricted.  
For example when $n=2^s, s\ge 4$, $P$ is parallelizable only when $m\in \{1,3,7\}$ and when $n=2^s-2, s\ge 5$, $m\in \{2,6\}.$  When $n=6, P$ is not parallelizable for any $m$.}
\end{remark}

\subsection{More examples of parallelizable generalized Dold manifolds} \label{paralleldold} 
We give examples of parallelizable manifolds $P(m,X)$ for some other classes of $X$.   Specifically, 
we take $X$ to be certain (i) Hopf manifold, (ii) complex torus, and (iii) compact 
Clifford-Klein form of a (non-compact) complex Lie group.  In all these case, it turns out that $\fix(\sigma)\ne \emptyset$ and 
$\hat\tau X\cong d\xi\oplus d\epsilon_\mathbb{R}$.  In particular $\textrm{span}(P(m,X))\ge d$.  If $2^{\varphi(m)}$ divides $(m+1+d)$, then $P(m,X)$ is stably parallelizable.  Furthermore, if $d>\rho(m+2d)$,  then $P(m,X)$ is parallelizable.  
 
(i) Let $\lambda >1$.  The infinite cyclic subgroup $\langle \lambda\rangle$ of the multiplicative group $\mathbb{R}^\times_{>0}$ 
acts on $\mathbb{C}^d_0:=\mathbb{C}^d\setminus \{0\}$ via scalar multiplication.  
Consider the Hopf manifold $X=X_\lambda:=\mathbb{C}^d_0/\langle \lambda\rangle$.  
Then $X\cong \mathbb{S}^1\times \mathbb{S}^{2d-1}$ is parallelizable.  
Although $X_\lambda$ is defined for any complex number $\lambda$ with $|\lambda| \ne 1$, our 
hypothesis that $\lambda$ is real implies that complex conjugation on $\mathbb{C}^d$ induces an involution 
$\sigma$ on $X$. Moreover $\fix(\sigma)=(\mathbb{R}^{d}\setminus \{0\})/\langle \lambda\rangle$  
is non-empty. 
In fact $\fix(\sigma)\cong \mathbb{S}^1\times \mathbb{S}^{d-1}$.  We claim that 
$\tau X$ is isomorphic to $d\epsilon_\mathbb{C}$ as a complex vector bundle.  Indeed, scalar multiplication 
$\lambda:\mathbb{C}^d_0\to \mathbb{C}^d_0$ induces multiplication by $\lambda$ on the tangent 
space $T_z\mathbb{C}^d_0$ for any $z\in \mathbb{C}^d_0$.  Therefore $TX=(\mathbb{C}_0^d\times \mathbb{C}^d)/\langle \lambda\rangle$ where $\langle \lambda\rangle $ acts diagonally.  
The required isomorphism $\phi: TX\to X\times \mathbb{C}^n$ is then obtained as 
$[z,v]\mapsto ([z], v/||z||)$.   We observe that this is well-defined since $\lambda$ is positive.  
Moreover, $\phi(T\sigma ([z,v]))= \phi( [\bar{z},\bar v])=([\bar z], \bar{v}/||z||)$.   
Thus $T\sigma$ corresponds to complex conjugation on $d\epsilon_\mathbb{C}$ and so $\hat{\tau} X\cong d\xi\oplus d\epsilon$ 
by Theorem \ref{stableparallel}(ii).   

(ii) Let $X=X_\Lambda\cong (\mathbb{S}^1)^{2d}$ be the complex torus $\mathbb{C}^d/\Lambda$ where 
$\Lambda\cong \mathbb{Z}^{2d}$ is 
stable under conjugation; equivalently $\Lambda=\Lambda_0+\sqrt{-1}\Lambda_0$ where $\Lambda_0$ is 
a lattice in $\mathbb{R}^d$.  Then complex conjugation on $\mathbb{C}^d$ induces a conjugation 
$\sigma$ on $X$. It is readily seen that $\fix(\sigma)=(\mathbb{R}^d+\frac{\sqrt{-1}}{2}\Lambda_0)/\Lambda_0$.  Also  
 $\tau X\cong d\epsilon_\mathbb{C}$ as a complex vector bundle.  As in (i) above, $\hat{\tau}X\cong d\xi\oplus d\epsilon_\mathbb{R}$.

(iii) More generally, suppose that $G\subset GL(N,\mathbb{C})$ is a connected complex linear Lie group 
such that $G$ is stable by conjugation $A\mapsto \bar A$ in $GL(n,\mathbb{C})$.  Suppose that 
$\Lambda$ a discrete subgroup of $G$ such that $X=G/\Lambda$ is compact; that is, $\Lambda$ is a uniform lattice in $G$.  
Assume that $\bar \Lambda =\Lambda$.   (For example, 
$G$ is the group of unipotent upper triangular matrices in $GL(N,\mathbb{C})$ with $\Gamma$ the subgroup of $G$ consisting 
matrices with entries in $\mathbb{Z}[\sqrt{-1}]$.)      
Then $X=G/\Lambda$ 
is {\it holomorphically parallelizable,} i.e., $\tau X$ is trivial as a complex analytic vector bundle.  See \cite{akhizer}.  In particular, 
$\tau X\cong d\epsilon_\mathbb{C}$.   Let $p:G\to X$ be the covering projection.
Denoting by $\mathfrak{g}$ the Lie algebra of $G$, viewed as the space of 
vector fields on $G$ invariant under right translation, we have a bundle isomorphism $f:X\times \mathfrak{g}\to TX$ defined as 
$(g\Gamma, V)\mapsto Tp_g(V_g)~\forall V\in \mathfrak{g}$.  This is well-defined since $V$ is invariant under right-translation.   
Under this isomorphism, $T\sigma$ is the standard $\sigma$-conjugation on $d\epsilon_\mathbb{C}$. 
So $\hat \tau X\cong d\xi\oplus d\epsilon_\mathbb{R}$.  
As the
identity coset is fixed by $\sigma$, $\fix(\sigma)\ne \emptyset$.

\subsection{Unoriented cobordism}   
Recall from the work of Thom and Pontrjagin (\cite[Ch. 4]{ms}) that the (unoriented) cobordism class of a smooth closed manifold is determined by its Stiefel-Whitney numbers.   Let $\sigma$ be a complex conjugation on a connected almost complex 
manifold $X$ and let $\dim_\mathbb{R} X=2d$.  Assume that $\fix(\sigma)\ne \emptyset$ and that 
$H^1(X;\mathbb{Z}_2)=0$. 
Proposition \ref{swformula1} allows us to compute certain Stiefel-Whitney 
numbers of $P(m,X)$ in terms of those of $X$, even without the knowledge of the cohomology algebra $H^*(P(m,X);\mathbb{Z}_2)$.  Let 
$s:\mathbb{R}P^m\to P(m,X)$ be the cross-section corresponding to an $x_0\in \fix(\sigma).$  We identify $\mathbb{R}P^m$ with 
its image under $s$ and $X$ with the fibre over $[e_{m+1}]\in \mathbb{R}P^m$.  
Then $X\cap \mathbb{R}P^m=\{[e_{m+1},x_0]\}$ 
and the intersection is transverse. Denoting the mod $2$ Poincar\'e dual of a submanifold 
$M\hookrightarrow P(m,X)$ by $[M]$, we have $[\mathbb{R}P^m]. [X]=[\mathbb{R}P^m\cap X]=
[\{[e_{m+1},x_0]\}]$, which is the generator of 
$H^{m+2d}(P(m,X);\mathbb{Z}_2)\cong \mathbb{Z}_2$.   
    
We claim that the 
class $[X]\in H^m(P(m,X);\mathbb{Z}_2)$ equals $x^m$.   To see this, let $S_j$ be the sphere 
$S_j=\{v\in \mathbb{S}^{m}\mid v\perp  e_j\}, 1\le j\le m$. 
and let $X_j$ 
be the submanifold $\{[v,x]\mid v\in S_j,x\in X\}\cong P(m-1,X)$.   Let $u_0=(e_1+\ldots +e_m)/\sqrt{m}$.  Then 
$C:=\{[\cos(t) u_0+\sin (t)e_{m+1}, x_0]\in P(m,X)\mid 0\le t\le \pi \}\cong \mathbb{R}P^1$ meets $X_j$ transversally at $[e_{m+1},x_0]$.   So $[C].[X_j]\ne 0$. 
It follows that $[X_j]=x, 1\le j\le m,$ since $H^1(P(m,X);\mathbb{Z}_2)=\mathbb{Z}_2x$.    
Also (i) $\cap_{1\le i<j}X_i$ intersects $X_j$ transversely for any $j\le m$,  and, (ii) 
$\cap_{1\le j\le m} X_j = X$.   
It follows that $[X]=[X_1]\cdots [X_m]=x^m$ as claimed.

Denote by $\mu_X, \mu_{P(m,X)}$ the mod $2$ fundamental 
classes of $X, P(m,X)$ respectively.  
Note that $w_{2j}(P(m,X))$ is of the form $w_{2j}(P(m,X))=\tilde{c}_j(X)+a_1x^2\tilde{c}_{j-1}(X)+\ldots+a_kx^{2k}\tilde{c}_{j-k}(X)$ for suitable $a_i\in\{0,1\}, 1\le i\le k,$
where $k=\min\{\lfloor m/2\rfloor , j\}$.   Similarly $w_{2j+1}(P(m,X))=b_0x\tilde{c}_j(X)+b_1x^3\tilde{c}_{j-1}(X)+\ldots+b_kx^{2k+1}\tilde{c}_{j-k}$, $b_i\in \{0,1\}, 0\le i\le k,$ with $k=\min\{\lfloor (m-1)/2\rfloor, j\}$.  A straightforward calculation using Theorem 
\ref{swformulaforp}  
reveals that $b_0=m+1+d-j.$  
Let $J=j_1,\dots, j_r$ be a sequence of positive integers with $|J|:=j_1+\cdots+j_r=m+2d$. 
Then $w_J(P(m,X)):=w_{j_1}(P(m,X))\ldots w_{j_r}(P(m,X))$
 is a polynomial in $x$ over the subring 
$\mathbb{Z}_2[\tilde{c}_1(X),\ldots,\tilde{c}_d(X)]\subset H^*(P(m,X);\mathbb{Z}_2)$.  
Since $x^{m+1}=0$, we see that $w_J(P(m,X))=0$ if the number of odd numbers among $j_k, 1\le k\le r,$ exceeds 
$m$.

Suppose that $I=i_1, \ldots, i_k; J=1^m.2I=1^m,2i_1,\ldots, 2i_k, $  (i.e., $j_t=1, 1\le t\le m$) and  
$P(m,X)$ is non-orientable,  so that $w_1(P(m,X))=x$, we 
have $w_J(P(m,X))=x^m.\tilde{c}_{I}(X)$.  Using $j^*(\tilde{c}_I(X))=c_I(X)=w_{2I}(X)$, we obtain that 
$w_J[P(m,X)]:=\langle w_J(P(m,X)),\mu_{P(m,X)}\rangle=\langle x^m.w_{2I}(P(m,X)), \mu_{P(m,X)}\rangle
=\langle w_{2I}(X),\mu_X\rangle=w_{2I}[X]\in \mathbb{Z}_2$.

\begin{theorem}\label{null-nonorient}  Suppose that $H^1(X;\mathbb{Z}_2)=0$ and that $ \fix(\sigma)\ne \emptyset$.\\
(i) Assume that $m\equiv d \mod 2$. If $[X]\ne 0$ in $\mathfrak{N}$, then $[P(m,X)]\ne 0$. \\
(ii) If $[P(1,X)]\ne 0$, then $[X]\ne 0$. 
\end{theorem}
\begin{proof}  (i) Since $m\equiv d\mod 2$, we have $w_1(P(m,X))=x$.  Since the odd Stiefel-Whitney 
classes $w_{2i+1}(X)$ vanish (as $X$ is an almost complex manifold), $[X]\ne 0$ implies that we must have 
that $w_{2I}[X]\ne 0$ for some $I$ with $|I|=d$.  
Then, by our above discussion $w_J[P(m,X)]\ne 0$ where $J=1^m.2I.$   This 
proves the first assertion.  

(ii)  Let $m=1$. $\dim P(1,X)=1+2d$ is odd.   Using $x^2=0$, we have, from the above discussion, that 
$w_{2j}(P(1,X))=\tilde{c}_j(X)$ and $w_{2j+1}(P(1,X))=(d-j)x\tilde c_j(X)$.   
Suppose that $w_J[P(1,X)]\ne 0$.  Then we see that exactly one term, say $j_k$, in $J$ must be odd. 
Write $j_k=2s+1$ where $s\ge 0$.    If $d-s$ is even, then $w_J[P(1,X)]=0$. So $d-s$ is odd and we have 
$w_J(P(1,X))=x \tilde{c}_I(X)$ where $2I$ is obtained from $J$ by replacing $j_k$ by $j_k-1$.  
Therefore $w_{2I}[X]=w_J[P(1,X)]\ne 0$.
This completes the proof.
\end{proof}

It remains to prove Theorem \ref{null-grassmann}.   The proof will involve finding an action of an elementary 
abelian $2$-group action on $P(m,\mathbb CG_{n,k})$ without stationary points.  In order to achieve this, we need to find   
certain units in a complex Clifford algebra $C^c_r$ which act on its simple modules as {\it real} transformations.  This is straightforward 
using the structure of real Clifford algebras $C_r, C'_r$ if $r=2p$, $p\equiv 1,3,4\mod 4$, but involves further considerations 
when $p\equiv 2\mod 4$.

\subsection*{Clifford algebras and their simple modules} 
We shall now recall the description and 
certain properties of real and complex Clifford algebras.  We refer the reader to \cite{husemoller} for 
details.

Let $C_r$ (resp. $C'_r$) be the Clifford algebra associated to 
$(\mathbb{R}^{r}, -||\cdot||^2)$ (resp. $(\mathbb R^r, ||\cdot||^2))$.    
Thus $C_r$ is generated as an $\mathbb R$-algebra by the 
elements $\phi_1,\cdots, \phi_r$ which satisfy the relations 
$\phi_i^2=-id~\forall i,$ and 
$\phi_i\circ\phi_j=-\phi_j\circ \phi_i, 1\le i<j\le r$.   Similarly $C'_r$ is generated as an $\mathbb R$-algebra by $\psi_1,\ldots, 
\psi_r$ which satisfy the relations $\psi_i^2=id ~\forall i,$ and $\psi_i\psi_j=-\psi_j\psi_i, 1\le i<j\le r$.  
We shall denote by $C^c_r$ the complex Clifford algebra $C_r\otimes_\mathbb{R}\mathbb{C}$.   
Note that $C_r^c\cong C'_r\otimes_\mathbb R\mathbb C$ under an isomorphism that sends $\phi_j$ to $\sqrt{-1}\psi_j$.  
 Following the notation in Husemoller's book \cite{husemoller},  we 
denote the matrix algebra $M_m(A)$ over a division ring $A$ by 
$A(m)$.  It is known that $C^c_r$ is isomorphic to $\mathbb C(2^p)$ or $\mathbb C(2^p)\times \mathbb C(2^p)$ according as 
$r=2p$ or $r=2p+1$.  

It is well known that $C_r, C'_r$ are isomorphic to algebras of the form $A(2^t)$ or $A(2^{s})\times A(2^{s})$
where $A=\mathbb R, \mathbb C,$ or the quaternions $\mathbb H$.  The values of $t, s$ are 
determined by comparing the dimensions.  Using the fact that $A\otimes_\mathbb R\mathbb C\cong \mathbb C, \mathbb C\times 
\mathbb C, \mathbb C(2)$ according as $A=\mathbb R, \mathbb C, \mathbb H$ respectively, it is readily seen that 
$C_r^c$ is isomorphic to one of the algebras $\mathbb C(2^{p})\times \mathbb 
C(2^{p})$ or $\mathbb C(2^p)$, according as $r=2p+1$ or $2p$ respectively. 

We consider $\mathbb C^{2^p}$ as a module over $C^c_r$ where $r=2p$.    
For our purposes, it is important to know 
whether the elements $\phi_i\in C_r^c, 1\le i\le r$, or $\psi_i\in C_r^c, 1\le i\le r$, act on $\mathbb C^{2^p}$ as {\it real} transformations, that is if the elements are matrices with {\it real} entries in $C^c_r=\mathbb C(2^p)$.   
This is guaranteed to be the case if at least one of the algebras $C_r$ or $C'_r$ is isomorphic to $\mathbb R(2^p)$.  
We have isomorphisms of $\mathbb R$-algebras 
$C_2'\cong \mathbb R(2), C_6\cong \mathbb R(8), 
C_8\cong \mathbb R(16)$.  Also, 
$C_{r+8}\cong C_r\otimes \mathbb R(16), C'_{r+8}\cong C'_r\otimes \mathbb R(16)$.  Since $\mathbb R(k)\otimes \mathbb R(l)=
\mathbb R(kl)$ and $\mathbb R(k)\otimes_\mathbb R\mathbb C\cong \mathbb C(k)$, using the isomorphism $C_r\otimes_\mathbb R\mathbb C\cong C^c_r\cong C'_r\otimes_\mathbb R\mathbb C$, we see that 
when $r\equiv 2 \mod 8$, the elements $\psi_i\in C^c_r, 1\le i\le r,$ are represented by 
real matrices and that when $r\equiv 6, 8\mod 8$, the same property holds for $\phi_i\in C^c_r, 1\le i\le r$.    
Therefore, we see that when $p$ is a 
positive integer such that  $p\equiv 3,4
\mod 4$ (resp. $p\equiv 1\mod 4$) $\mathbb C^{2^p}$ has the structure of a simple $C^c_{2p}$-module on which 
$\phi_i , 1\le i\le 2p,$ (resp. $\psi_i, 1\le i\le 2p$) acts as real transformations, that is, via matrices with real entries.  
  
Let $p\equiv 2\mod 4$.  The real Clifford algebras $C_r, C_r'$ are {\it not} matrix algebras over the reals when $r=2p$ or $2p+1$.   
So we proceed as follows.
Write $r=2p=8q+4$.  We have the isomorphisms $C_{8q+2}'\cong \mathbb R(2^{4q+1})$ with its generators $\psi_i, 1\le i\le r-2$.  
Consider the $\mathbb R$-algebra $C$ generated by 
the elements $\theta_i, 1\le i\le r,$ expressed as $2\times 2$ block matrix with block sizes $p$ as follows:
\[\theta_i=\left \{ 
\begin{array}{ll} 
\bigl( \begin{smallmatrix} 0& \psi_i\\ -\psi_i& 0\end{smallmatrix}\bigr), & 1\le i\le r-2,\\
\bigl(\begin{smallmatrix} I&0\\0&-I\end{smallmatrix} \bigr), & i=r-1,\\
\big(\begin{smallmatrix} 0&I\\I&0\end{smallmatrix} \bigr), & i=r.\\
\end{array}
\right.
\]
Then the following relations are readily verified: (i) $\theta_i\theta_j=-\theta_j\theta_i$ if $1\le i<j\le r$, and, 
(ii) $\theta^2_i=-1$ if $1\le i\le r-2$ and $\theta_i^2=1$ if $i=r-1,r$.  Moreover, it is easily verified that 
$\mathbb R$-algebra generated by the $\theta_i$ 
equals $\mathbb R(2^p)$.  \footnote{Thus 
$C$ is the real Clifford algebra associated to the indefinite (non-degenerate) quadratic form with signature 
$(2,r-2)$. See \cite[ Chapter 13]{porteous}.}
Therefore $C\otimes_\mathbb R\mathbb C=\mathbb C(2^p)\cong C^c_r$.  
In particular, the elements $\theta_i, 1\le i\le r,$ act as real transformations on the simple module $\mathbb C^{2^p}$ of $C^c_r$.

\noindent
   {\bf Notation:}  
For $1\le i\le r$, we shall denote by $\theta_i\in C^c_r$ the element $\psi_i $ 
(resp. $\phi_i$) when $r\equiv 2 \mod 8$ (resp. $r\equiv 6,8\mod 8$).   When $r\equiv 4\mod 8$, the $\theta_i\in C^c_r$ are as defined 
above. 

The above discussion establishes the validity of the following lemma.  

\begin{lemma} \label{clifford1} Let $r=2p$ be any even positive number.  With the above notations, the elements $\theta_i\in C^c_r\cong \mathbb C(2^p), 1\le i\le r,$ satisfy the following conditions:\\
(i) $\theta_i\theta_j=-\theta_j\theta_i, i\ne j$ and $\theta_i^2=\pm 1$ for $i\le r$, \\
(ii) the $\mathbb R$-subalgebra of $C^c_r$ generated by $\theta_i, 1\le i\le r,$ is isomorphic to $\mathbb R(2^p)$,\\
(iii) the $\theta_i\in C^c_r$ act as a real transformation on the simple $C^c_r$ module $\mathbb C^{2^p}.$
 \hfill $\Box$
\end{lemma}

We are now ready to prove Theorem \ref{null-grassmann}.

{\it Proof of Theorem \ref{null-grassmann}}: (i).  Write $n=2^p n_0 $ where $n_0$ is odd and $p\ge 1$.  Suppose that $2^p$ does not divide $k$.

Now let $r=2p$.   We regard $\mathbb C^n$ as a sum of $n_0$ copies of the 
simple $C_r^c$-module $\mathbb C^{2^p}$.  
With notations as in Lemma \ref{clifford1}, 
let $t_i, 1\le i\le r$, denote the 
smooth map of the complex Grassmann manifold $\mathbb CG_{n,k}$ defined as $V\mapsto \theta_i(V), 1\le i\le r$. 
Then $t_i^2=id$ for $i\le r$ since $\theta^2_i=\pm 1$. Also $t_it_j=t_jt_i$ for $1\le i<j\le r$ since $\theta_i\theta_j=-\theta_j\theta_i$.  
So, the $t_i$ define a smooth action of the group $(\mathbb Z/2\mathbb Z)^r$.   Any stationary point $V$ of this action 
is a complex vector space of dimension $k$ such that $\theta_i(V)~\forall i\le r$.  This means that $V$ is a module of over the 
$\mathbb C$-algebra generated by the $\theta_i, 1\le i\le r$, that is, $V$ is a $C^c_r$-module.  In particular 
the $(\mathbb Z/2\mathbb Z)^r$-action on $\mathbb CG_{n,k}$ is stationary point free since $k$ is not divisible by $2^p$.  

The fact that the $\theta_i$ are real transformations implies that the $t_i$ commute with complex conjugation $\sigma$, defined as 
$\sigma(V)= \overline V$.  This means that the $t_i$ define an involution, again denoted $t_i$, on the 
generalized Dold manifold $P(m, \mathbb CG_{n,k}).$  Explicitly, $t_i([u, V])=[u, t_i(V)]$ is meaningful 
since $(-u, t_i(\overline V))
=(-u, \overline{t_i(V)})\sim (u,t_i(V))$.  

We claim that the action of $(\mathbb Z/2\mathbb Z)^r$ has no stationary points.  Indeed, $[u,V]=t_i([u,V])=[u, t_i(V)]$ 
implies that $t_i(V)=V$ and so if $[u, V]\in P(m,\mathbb CG_{n,k})$ is a stationary point, then $V\in \mathbb CG_{n,k}$ would be a 
stationary point, contrary to what was just observed.  
Now, by \cite[Theorem 30.1]{cf}, it follows that $[P(m,X)]=0$. 

(ii) Suppose that $\nu_2(n)=\nu_2(k)$.  
 Then $[\mathbb{C}G_{n,k}]\ne 0$ by the main theorem of \cite{sankaran}.  (See also \cite{sankaran1}.) 
Note that 
$\dim_\mathbb{C}\mathbb{C}G_{n,k}$ is even in this case.  If $m$ is also even, then it follows that 
$[P(m,\mathbb{C}G_{n,k})]\ne 0$ by  
Theorem \ref{null-nonorient}(i).  \hfill $\Box$

\begin{remark}\label{null-flag}{\em
It appears to be unknown precisely which (real or complex) flag manifolds are unoriented boundaries.  
Let $n_1,\ldots, n_r\ge 1$ be 
integers and let $n=\sum_{1\le j\le r} n_j$.  
Proceeding as in the case of the $P(m,\mathbb C G_{n,k})$ it is readily seen 
that $[\mathbb{C}G(n_1,\ldots,n_r)]$ and
$[P(m;n_1,\ldots,n_r)]$ in $\frak{N}$ are zero if $\nu_2(n)>\nu_2(n_j)$ for some $j$.  Also, if $n_i=n_j$ for some $i\ne j$, then 
$X:=\mathbb{C}G(n_1,\ldots, n_r)$ admits a fixed point free involution $t_{i,j}$, which swaps the $i$-th  and the $j$-component 
of each flag ${\bf L}$ in $X$.  Clearly $t_{i,j}(\bar {\bf L})=\overline{t_{i,j}({\bf L})}, {\bf L}\in X,$ and so we obtain 
an involution $[v, {\bf L}]\mapsto [v, t_{i,j}({\bf L})]$  on $P(m;n_1,\ldots,n_r)$, which is again fixed point free. 
It follows that $[P(m;n_1,\ldots, n_r)]=0$ in this case. 
If $m\equiv d\mod 2$ where $d=\dim_\mathbb{C} X=\sum_{1\le i<j\le r}n_in_j$ and 
if $[X]\ne 0$, then $[P(m;n_1,\ldots, n_r)]\ne 0$ by Theorem \ref{null-nonorient}.  For example, it is 
known that $\chi(X)= n!/(n_1!.\ldots. n_r!)$.  So 
if $m$ and $d$ are even and if $n!/(n_1!.\ldots. n_r!)$ 
is odd, then $\chi(P(m;n_1,\ldots,n_r))$ is also odd and so $[P(m;n_1,\ldots,n_r)]\ne 0$. 
}
\end{remark}
\noindent
{\bf Acknowldegments:}  Sankaran thanks Peter Zvengrowski for bringing to his attention the papers of J\'ulius Korba\v s \cite{korbas} and Peter Novotn\'y \cite{novotny}.






\begin{thebibliography}{99}
\bibitem{adams} Adams, J. F.  \textit{Vector fields on spheres.} Ann.  Math.  \textbf{75}, (1962), 603--632.
\bibitem{akhizer} Akhiezer, D. N.  {\it Homogeneous complex manifolds.} Several complex variables-IV, 
Translation edited by S. G. Gindikin and G. M. Khenkin. 195--244.
Encycl. Math. Sci. {\bf 10} Springer, New York, 1990. 

\bibitem{bk} Bredon, G. E.; Kosi\'nski, A.
{\it Vector fields on $\pi$-manifolds.}
Ann. Math. (2) {\bf 84} (1966) 85--90. 

\bibitem{ct}  Chakraborty, Prateep; Thakur, Ajay Singh {\it Nonexistence of almost complex structures on the product} $S^{2m}\times M$. Topology Appl. {\bf 199} (2016), 102--110. 

\bibitem{cf} Conner, P. E.; Floyd, E. E. Differentiable periodic maps.   Ergebnisse der Mathematik und 
Ihrer Grenzgebiete. {\bf 33} Springer-Verlag, Berlin, 1963. 



\bibitem{dold} Dold, Albrecht {\it Erzeugende der Thomschen Algebra} $\frak{N}$. Math. Zeit. {\bf 65} (1956) 25--35.

\bibitem{husemoller} Husemoller, D. Fibre bundles. 
Third Edition, Grad. Texts in Math. \textbf{20}, Springer-Verlag, N.Y. 1994.

\bibitem{korbas85} 
 Korba\v{s}, J\'{u}lius {\it Vector fields on real flag manifolds.} Ann. Global Anal. Geom. {\bf 3} (1985), no. 2, 173--184.

\bibitem{korbas} Korba\v{s}, J\'ulius {\it On the parallelizability and span of Dold manifolds.} Proc. Amer. Math. Soc. 
{\bf 141} (2013) 2933--2939.
\bibitem{koschorke}  Koschorke, U. Vector fields and other vector bundle morphisms--a singularity approach. 
Lecture Notes in Mathematics, \textbf{847}, Springer, Berlin, 1981.
\bibitem{lam} Lam, K.-Y. \textit{A formula for the tangent bundle of flag manifolds and related manifolds},.Trans. Amer. Math. Soc.  \textbf{213}, (1975), 305--314. 

\bibitem{li} Li, Bang He {\it Codimension $1$ and $2$ imbeddings of Dold manifolds}.   Kexue Tongbao (English Ed.) {\bf 33} (1988), no. 3, 182--185.
\bibitem{ms} Milnor, J. W.; Stasheff, J. D.   Characteristic classes. 
Annals of Mathematics Studies, {\bf 76}, Princeton University Press, Princeton, N. J. 1974.
\bibitem{nt}  Naolekar, Aniruddha C.; Thakur, Ajay Singh {\it Note on the characteristic rank of vector bundles.} Math. Slovaca {\bf 64} (2014), no. 6, 1525--1540.
\bibitem{novotny}
Novotn\'{y}, P. {\it Span of Dold manifolds.}  Bull. Belg. Math. Soc. Simon Stevin, {\bf15} (2008), 687--698. 
\bibitem{porteous} Porteous, Ian R. Topological geometry. Van Nostrand Reinhold Co., London, 1969.
\bibitem{sankaran1} Sankaran, P. 
{\it Which Grassmannians bound?} Arch. Math. (Basel) {\bf 50} (1988),  474--476.
\bibitem{sankaran} Sankaran, P.  {\it 
Determination of Grassmann manifolds which are boundaries.} Canad. Math. Bull. {\bf 34} (1991),  119--122.
\bibitem{sz} Sankaran, P.; Zvengrowski, P. {\it On stable parallelizability of flag manifolds.}
Pacific J. Math. {\bf 122} (1986), no. 2, 455--458. 
\bibitem{spanier} Spanier, Edwin  H. Algebraic topology.  Corrected reprint. Springer-Verlag, New York, 1981. 
\bibitem{thakur}  Thakur, Ajay Singh {\it On trivialities of Stiefel-Whitney classes of vector bundles over iterated 
suspensions of Dold manifolds.} Homology Homotopy Appl. {\bf 15} (2013), no. 1, 223--233. 
\bibitem{ucci} 
 Ucci, J. J. {\it Immersions and embeddings of Dold manifolds.} Topology {\bf 4} (1965) 283--293.

\end{thebibliography}
\end{document}